\documentclass{article}

\usepackage{stmaryrd}
\usepackage{amsmath}
\usepackage{amssymb}
\usepackage{amsfonts}
\usepackage{graphicx}
\usepackage{subfigure}

\title{The Cantor-Bernstein-Schr\"oder theorem\\ in algebra}

\author{{\sc Hector Freytes}}

\date{{\small
1. Dipartimento di Filosofia, Universit\'a di Cagliari\\ Via Is Mirionis I, Cagliari - Italia.\\
2. Departamento de Matem\'{a}tica UNR-CONICET \\ Av. Pellegrini 250, CP 2000 Rosario - Argentina.\\
email: hfreytes@gmail.com}}

\begin{document}

\bibliographystyle{plain}

\maketitle

\begin{abstract}
\noindent
The famous Cantor-Bernstein-Schr\"oder theorem (CBS-theorem for short) of set theory was generalized by Sikorski and Tarski to $\sigma$-complete Boolean algebras. After this, 
numerous generalizations of the CBS-theorem, extending the Sikorski-Tarski version to different classes of algebras, have been established. Among these classes there are lattice ordered groups, orthomodular lattices, MV-algebras, residuated lattices, etc. This suggests to consider a common algebraic framework in which various versions of the CBS-theorem can be formulated. In this work we provide this algebraic framework establishing necessary and sufficient conditions for the validity of the theorem. We also show how this abstract framework includes the versions of the CBS-theorem already present in the literature as well as new versions of the theorem extended to other classes such as groups, modules, semigroups, rings, $*$-rings etc.  
\end{abstract}

\begin{small}



\end{small}

\bibliography{pom}

\newtheorem{theo}{Theorem}[section]

\newtheorem{definition}[theo]{Definition}

\newtheorem{lem}[theo]{Lemma}

\newtheorem{met}[theo]{Method}

\newtheorem{prop}[theo]{Proposition}

\newtheorem{coro}[theo]{Corollary}

\newtheorem{exam}[theo]{Example}

\newtheorem{rema}[theo]{Remark}{\hspace*{4mm}}

\newtheorem{example}[theo]{Example}

\newcommand{\proof}{\noindent {\em Proof:\/}{\hspace*{4mm}}}

\newcommand{\qed}{\hfill$\Box$}

\newcommand{\ninv}{\mathord{\sim}} 

\section*{Introduction}\label{INTRO}

The famous Cantor-Bernstein-Schr\"oder  theorem of the set theory states that, 
\begin{enumerate}
\item[] ``{\it if a set $X$ can be embedded into a set $Y$ and vice versa, then there is a one-to-one function of $X$ onto $Y$}''. 
\end{enumerate}
 The history of this theorem is rather curious. The earliest record of the theorem might be a letter to Dedekind dated 5 november 1882 where Cantor conjectured the theorem. Dedekind proved it in 1887 but did not publish it. His proof was printed only in his collected works in 1932.  Schr\"oder proved the theorem in 1894 but he published it in 1898 \cite{SCH1, SCH2}. However Schr\"oder's proof was defective. Korselt wrote to Schr\"oder about the error in 1902 and few weeks later he sent a proof of the theorem to the Mathematische Annalen. Korselt paper appeared in 1911 \cite{KOR}. Bernstein, a 19 years old Cantor student, proved the theorem. His proof find its way to the public through  Borel because Cantor showed the proof to Borel in the 1897 during the International Congress of Mathematicians in Z\"{u}rich. The Bernstein proof was published in 1898 in the appendix of a Borel book \cite{Borel} and in 1901 Bernstein's thesis appeared with his proof.  Several years later, at the end of the forties, Sikorski \cite{Si} and independently Tarski \cite{Ta}, showed that the $CBS$-theorem is a particular case of a statement on $\sigma$-complete Boolean algebras. Following this idea, several authors have extended the Sikorski-Tarski version to classes of algebras more general than Boolean algebras. Among these classes there are
lattice ordered groups \cite{J2},  $MV$-algebras \cite{DMN,J0}, orthomodular lattices \cite{DMP}, effect algebras \cite{JEN}, pseudo effect algebras \cite{DVU4}, pseudo $MV$-algebras \cite{J}, pseudo $BCK$-algebras \cite{KU1} and in general, algebras with an underlying lattice structure such that the central elements of this lattice determine a direct decomposition of the algebra \cite{FR}. It suggests that the CBS-theorem can be formulated in a common algebraic framework from which all the versions of the theorem mentioned above stem. Moreover, this framework  also yields new versions of the CBS-theorem, applied to other algebraic structures. 

In this work we provide a such common algebraic framework. It consists of a category ${\cal A}$ of algebras of the same type and a presheaf, called {\it congruences presheaf}, acting on the congruence lattice of each algebra of ${\cal A}$. In this way, necessary and sufficient conditions for the validity of the CBS-theorem are given in terms of properties that certain algebras in ${\cal A}$ should satisfy with respect to the congruence presheaf.

The paper is structured as follows. Section \ref{Basic Notions} contains generalities on lattice theory, universal algebra and some technical results that are used in subsequent sections. In Section \ref{AFCBS} the crucial notion of congruences presheaf is introduced and the abstract framework for the CBS-theorem is provided. A necessary and sufficient condition for the validity of the CBS-theorem is given. Quasi-cyclic groups are studied as an example of algebras satisfying the CBS-theorem. In Section \ref{FCPRES} a congruences presheaf related to factor congruences is introduced and a CBS-theorem respect to this special presheaf is established. Injective modules and divisible groups  are studied as examples of algebras satisfying the CBS-theorem. A useful  necessary and sufficient condition for the validity of the CBS-theorem, restricted to this particular congruences presheaf, is also provided. In Section \ref{BFCFC} our abstract version of the CBS-theorem is studied in categories of algebras having Boolean factor congruences (BFC). This particular framework allow us to consider  versions of the theorem extended to algebras with an underlying lattice structure as lattice ordered groups, orthomodular lattices, residuated lattices, \L ukasiewicz and Post algebras; semigroups with $0,1$, bounded semilattices, commutative pseudo BCK-algebras, rings with unity, $*$-rings etc. Finally, we extend our abstract framework to two categories of algebras defined by partial operations.

\section{Basic Notions}\label{Basic Notions}
We recall from \cite{Bir, Bur, MM} some notions basic about lattice theory and universal algebra that will play an important role in what follows. Let $\langle L, \leq \rangle$ be an ordered set.  An interval {\rm [{\it a,b}]} of $L$ is defined as the set $\{\,x\in A : a\leq x\leq b \}$. $L$ is called {\it bounded} provided it has a smallest element $0$ and a greatest element $1$. Let $L$ be a bounded ordered set. A subset $X$ of $L$ is {\it orthogonal} ({\it dual orthogonal}) iff $x\land y = 0$ ($x\lor y = 1$) whenever $x,y$ are distinct elements of $X$. 

Let $\langle L,\lor,\land\rangle$ be a lattice.  Given $a, b, c$ in $L$, we write: $(a,b,c)D$  iff $(a\lor b)\land c = (a\land c)\lor (b\land c)$; $(a,b,c)D^{*}$ iff $(a\land b)\lor c = (a\lor c)\land (b\lor c)$ and, $(a,b,c)T$   iff $(a,b,c)D$, (a,b,c)$D^{*}$ hold for all permutations of $a, b, c$. An element $z$ of a lattice $L$ is called a {\it neutral element} iff for all elements $a,b\in L$ we have\ $(a,b,z)T$. The lattice $L$ is {\it$\sigma$-complete} iff $L$ admits denumerable supremum and denumerable infimum. In particular, $L$ is said to be {\it orthogonally $\sigma$-complete} ({\it dual orthogonal $\sigma$-complete}) iff every denumerable orthogonal (dual orthogonal) subset of $L$ has supremum (infimum) in $L$. 

Let $\langle L,\lor,\land,0,1\rangle$ be a bounded lattice. A {\it complement} of an element $a\in L$ is an element $\neg a\in L$ such that $a\lor \neg a = 1$ and $a\land \neg a = 0$. $L$ is called {\it complemented} when every element of $L$ has a complement. $L$ is a {\it Boolean algebra} iff it is a complemented distributive lattice. If $L$ is a Boolean algebra then the complement of an element $a$ is unique. An element $z \in L$ is called a {\it central element} iff $z$ is a neutral element having a complement. The set of all central elements of $L$ is called the {\it center} of $L$ and it is denoted by $Z(L)$. The center $Z(L)$ is a Boolean sublattice of $L$ \cite[Theorem 4.15]{MM}.

\begin{prop}\label{CENTER}{\rm  \cite[Proposition 3.1]{FR}}
Let $L$ be a bounded lattice and $z\in Z(L)$. Then,
\begin{enumerate}

\item
$Z(L)\cap [z,1] = Z([z,1])$,

\item
$[z,1]_L = \langle [z,1], \lor, \land, \neg_z, z, 1 \rangle$ where $\neg_z x = z \lor \neg x $, is a Boolean algebra. 
\end{enumerate}

\qed
\end{prop}

Let $\tau$ be a type of algebras and $X$ be a denumerable set of variables such that $\tau \cap X = \emptyset$. We denote by $Term_\tau(X)$ the set of terms built from the set of variables $X$. Each element $t \in Term_\tau(X)$ is referred as a {\it $\tau$-term}. For a $\tau$-term $t$ we often write $t(x_1,x_2 \ldots x_n)$ to indicate that the variables occurring in $t$ are among $x_1,x_2 \ldots x_n$. If $t\in Term_\tau(X)$ and $A$ is an algebra of type $\tau$ then we denote by $t^A$ the interpretation of $t$ in the algebra $A$. A {\it $\tau$-homomorphism} is a function between algebras of type $\tau$ preserving the $\tau$-operations. We write $A\cong_\tau B$ to indicate that there exists a $\tau$-isomorphism between the algebras $A$ and $B$ of type $\tau$.  An equation of type $\tau$ is an expression of the form $s=t$ such that $s,t \in  Term_\tau(X)$ and the symbol $=$ is interpreted as the identity. A quasi equation is an expression of the form $(\&_{i=1}^{n}s_i = t_i) \Longrightarrow s = t$ where $t_i, s_i, s, t \in  Term_\tau(X)$ and $\&_{i=1}^{n}$ denotes a logical $n$-conjunction. 

Let ${\cal A}$ be a class of algebras of type $\tau$. The {\it language of ${\cal A}$} is the first order language with identity built from the set $Term_\tau(X)$. If $\Phi$ is a sentence in the language of ${\cal A}$ and $A \in {\cal A}$ then, $A \models \Phi$ means that {\it $\Phi$ holds in the algebra $A$}. The sentence {\it $\Phi$ holds in the class ${\cal A}$}, abbreviated as ${\cal A} \models \Phi$, iff for each $A\in {\cal A}$, $A \models \Phi$. If $\Sigma$ is a set of sentences in the language of ${\cal A}$ then, $A \models \Sigma$ means that $A \models \Phi$ for each $\Phi \in \Sigma$. The class ${\cal A}$ is a {\it variety} ({\it quasivariety})  iff there exists a set $\Sigma$ of equations  (quasi equations) in the language of ${\cal A}$ such that ${\cal A} = \{A: A \models \Sigma\} $. Equivalently, ${\cal A}$ is a variety iff it is closed under homomorphic images -i.e. quotient algebras-, subalgebras and direct products. While ${\cal A}$ is a quasivariety iff it is closed by subalgebras, direct products but not necessarily under homomorphic images.

Let $A$ be an algebra of type $\tau$. We denote by $Con(A)$ the congruence lattice of $A$. The largest congruence on $A$, given by $A^2$, is denoted by $\nabla_A$ and the smallest, given by the diagonal $\{(a,a):a\in A\}$, is denoted by $\Delta_A$. For $a\in A$ and  $\theta \in Con (A)$, $a_\theta$ denotes the congruence class of $a$ modulo $\theta$. 
Let $\theta_1, \theta_2 \in Con(A)$. Then we say that $\theta_1, \theta_2$ are {\it permutable} iff $\theta_1 \circ \theta_2 = \theta_2 \circ \theta_1$ where $\circ$ is the relational product defined as $\theta_1 \circ \theta_2 = \{(x,y) \in A^2: \exists w\in A, \mbox{with $(x,w) \in \theta_1$ and $(w,y) \in \theta_2$}\}$. In \cite[Theorem 5.9]{Bur} it is proved that
$\theta_1, \theta_2$ are permutable iff $\theta \lor \sigma = \theta \circ \sigma$. Let $\sigma \in Con(A)$. If $\theta \in [\sigma, \nabla_A]$ then, 
\begin{equation}\label{QUOTCONG}
\theta / \sigma = \{(x_\sigma, y_\sigma) \in (A/\sigma)^2: (x,y) \in \theta \} 
\end{equation}
is a congruence on $A/\sigma$. The following results will be quite important in the next sections:

\begin{theo}\label{CORRESP}
Let $A$ be an algebra of type $\tau$ and $\sigma \in Con(A)$. Then:

\begin{enumerate}

\item
If $\sigma \subseteq \theta$ then, $f:(A/\sigma)/(\theta/\sigma) \rightarrow (A/\theta)$ such that $f({a_{\sigma}}_{(\theta/\sigma)}) = a_{\theta}$ is a $\tau$-isomorphism.  

\item
$u_\sigma: [\sigma, \nabla_A]\rightarrow Con(A/\sigma)$ such that $u_\sigma(\theta)= \theta /\sigma$ is a lattice isomorphism.

\item
If $\sigma \subseteq \theta_1$ and $\sigma \subseteq \theta_2$ then, $(a,b) \in \theta_1 \circ \theta_2$ iff  $(a_\sigma,b_\sigma) \in \theta_1/\sigma \circ \theta_2/\sigma$.

\end{enumerate}

\end{theo}

\begin{proof}
1) See {\rm  \cite[Theorem 6.15]{Bur}}. 2) See {\rm \cite[Theorem 6.20]{Bur}}. 3) $(a,b) \in \theta_1 \circ \theta_2$ iff there exists $c\in A$ such that $(a,c) \in \theta_1$ and $(c,b) \in  \theta_2 $ iff $(a_\sigma,c_\sigma) \in \theta_1/\sigma$ and $(c_\sigma,b_\sigma) \in \theta_2/\sigma$ iff $(a_\sigma, b_\sigma) \in \theta_1/\sigma \circ \theta_2/\sigma$.

\qed
\end{proof}

A congruence $\theta$ on $A$ is a {\it factor congruence} iff there exists $\neg \theta \in Con(A)$, referred to as a {\it a factor complement of $\theta$}, such that $\theta \cap \neg \theta = \Delta_A$, $\theta \lor \neg \theta = \nabla_A$ and $\theta$ permutes with $\neg \theta$ (or equivalently, by \cite[Theorem 5.9]{Bur}, $\theta \cap \neg \theta = \Delta_A$ and $\theta \circ \neg \theta = \nabla_A$). In this case $A$ is $\tau$-isomorphic to $A/\theta \times A/ \neg \theta$. The pair $(\theta, \neg \theta)$ is called a {\it pair of factor congruences}.  We denote by $FC(A)$ the set of factor congruences on $A$.

\begin{prop}\label{BFCQUOT}
Let $A$ be an algebra of type $\tau$, $\sigma \in FC(A)$ and $\theta \in [\sigma, \nabla_A]$ such that $\theta/\sigma \in FC(A/\sigma)$. Then $\theta \in FC(A)$.

\end{prop}

\begin{proof}
Let us suppose that $(\sigma, \neg \sigma)$ is a pair of factor congruences in $FC(A)$ and $(\theta/\sigma, \neg (\theta/\sigma))$ is a pair of factor congruences in $FC(A/\sigma)$. Then, by Theorem \ref{CORRESP}-1, we have that  
\begin{eqnarray*}
A & \cong_\tau & A/\sigma \times A/\neg \sigma \cong_\tau ((A/\sigma)/(\theta/\sigma) \times (A/\sigma)/\neg(\theta/\sigma)) \times A/\neg \sigma\\
& \cong_\tau & A/\theta \times B
\end{eqnarray*}
where $B = (A/\sigma)/\neg(\theta/\sigma) \times A/\neg \sigma$. Consider the diagram $A \stackrel{f}{\rightarrow} A/\theta \times B \stackrel{\pi_B}{\rightarrow} B$
where $f$ is a $\tau$-isomorphism. Then, $(\theta, Ker(\pi_B f))$ is a pair of factor congruences on $A$ proving that $\theta \in FC(A)$.

\qed
\end{proof}

\begin{prop}\label{NDIRECTPRODUCT}
Let $A$ be an algebra of type $\tau$ and let us consider the denumerable direct product $B = \prod_{{\mathbb N}}A$.  Then there exists $\sigma \in FC(B)$ such that $B \cong_{\tau} B/\sigma$. 

\end{prop}

\begin{proof}
If we consider $B = A \times \prod_{i\in  {\mathbb N} - \{1\}}A$ then, $f:B \rightarrow \prod_{i\in  {\mathbb N} - \{1\}}A$ defined by
$B \ni (b_i)_{i\in  {\mathbb N}} \mapsto f((b_i)_{i\in  {\mathbb N}})= (a_i)_{i\geq 2}$ where $a_2 = b_1, a_3 = b_2, \ldots a_{n+1} = b_n, \ldots$ is a ${\tau}$-isomorphism. Then, by defining $\sigma = Ker(\pi{_{\prod_{i\in  {\mathbb N} - \{1\}}  A}})$, we have that $\sigma \in FC(B)$ and $B \cong_{\tau} B/\sigma$.

\qed
\end{proof}

\begin{definition}
{\rm A {\it category of algebras} is a category ${\cal A}$ whose objects are algebras of type $\tau$ and whose morphisms are all of the $\tau$-homomorphisms (also called ${\cal A}$-homomorphisms) $f:A \rightarrow B$ such that $A,B$ are objects of ${\cal A}$.
}
\end{definition}

Let ${\cal A}$ be a category of algebras. We denote by $Ob({\cal A})$ the class of objects of ${\cal A}$ and by $Hom_{\cal A}$ the set of all ${\cal A}$-homomorphisms.  
For the sake of simplicity if $A$ is an object of ${\cal A}$  then we write $A \in {\cal A}$ when there is no confusion. If two objects $A,B \in {\cal A}$ are $\tau$-isomorphics then we denote it by $A\cong_{_{\cal A}}B$. Note that if ${\cal A}$ a class of algebras of type $\tau$ then we can identify ${\cal A}$ with a category of algebras by considering all of the $\tau$-homomorphisms between algebras of ${\cal A}$ as arrows of ${\cal A}$.  In this sense varieties and quasivarieties can be seen as categories of algebras. A {\it presheaf} on a category ${\cal C}$ is  a functor ${\cal F}: {\cal C}^{op} \rightarrow Set$ where ${\cal C}^{op}$ is the dual category of ${\cal C}$ and $Set$ is the category of all sets.

\section{Presheaf approach to the CBS-theorem } \label{AFCBS}

In this section we provide a general  framework for the CBS-theorem that capture the Sikorski-Tarski version in purely algebraic terms. Taking into account Proposition \ref{CENTER}-2, the Sikorski-Tarski version of the CBS-theorem reads as follows: 

\begin{theo}
For any two $\sigma$-complete Boolean algebras $A$, $B$ and for any element $a\in A$, and $b\in B$, if $A$ is Boolean-isomorphic to $[b,1]_B$ and $B$ is Boolean-isomorphic to $[a,1]_A$, then $A$ is Boolean-isomorphic to $B$. \qed
\end{theo}

Clearly, to obtain the classical CBS-theorem it suffices to assume that $A$ and $B$ are the power sets of two sets endowed with  the natural set-theoretic Boolean operations. In this algebraic version of the CBS-theorem, we first note that $[a,1]_A$ and $[b,1]_B$ are Boolean-isomorphics to the quotients algebras $A/\theta_a$ and $B/\theta_b$ respectively where,  $\theta_a = \{(x,y) \in A^2: x\lor a = y \lor a \} \in FC(A)$ and  $\theta_b = \{(x,y) \in B^2: x\lor b = y \lor b \} \in FC(B)$. We also note that the hypothesis of $\sigma$-completeness in $A$ and $B$ can be equivalent expressed as a condition of $\sigma$-completeness in $FC(A)$ and $FC(B)$. These facts suggest that, the Sikorski-Tarski version of the CBS-theorem depends on order-theoretic properties about the set of factor congruences of Boolean algebras. This framework can become even more general by taking into account the following:

A category of algebras ${\cal A}$ where for each $A\in {\cal A}$ a subset ${\cal K}(A) \subseteq Con(A)$ is considered. Order-theoretic properties imposed on the set ${\cal K}(A)$,  will allow us to establish conditions for the validity of the $CBS$-theorem formulated in this abstract framework. In order to do it, we need some preliminary results. 

Let $A,B$ two algebras of type $\tau$ and $f:A\rightarrow B$ be a $\tau$-homomorphism. Then we define the following sets: 
\begin{equation}\label{UPF}
f^*(\theta) = \{(a,b) \in A^2: (f(a),f(b)) \in \theta\} \hspace{0.2cm} \mbox{for each $\theta \in Con(B)$}.
\end{equation}
\begin{equation}\label{DOWNF}
f_*(\theta) = \{(f(a), f(b)) \in B^2 : (a,b) \in \theta\} \hspace{0.2cm} \mbox{for each $\theta \in Con(A)$}.
\end{equation}

\begin{prop}\label{AUX}
Let $A,B$ be two algebras of type $\tau$ and $f:A\rightarrow B$ be a $\tau$-homomorphism. Then we have:

\begin{enumerate}
\item
The assignment $Con(B)\ni\theta \mapsto f^*(\theta)$ defines  an order homomorphism $f^*:Con(B) \rightarrow Con(A)$. 

\item
$(gf)^* = f^*g^*$ whenever the composition of $\tau$-homomorphism $gf$ is defined.

\item
$1_A^* = 1_{Con(A)} $.   

\item
If $f$ is a $\tau$-isomorphism then, the assignment $Con(A) \ni \theta \mapsto f_*(\theta)$ defines an order isomorphism $f_*:Con(A) \rightarrow Con(B)$ and $f_* = (f^*)^{-1} = (f^{-1})^*$. Moreover $f': A/\theta \rightarrow B/f_*(\theta)$ such that $f'(x_\theta) = f(x)_{f_*(\theta)}$ is a $\tau$-isomorphism.

\item
If $f$ is a $\tau$-isomorphism and $\theta_1, \theta _2 \in Con(A)$ are permutable then $f_*(\theta_1)$, $f_*(\theta_2)$ are permutable in $Con(B)$   

\end{enumerate}

\end{prop}

\begin{proof}
 1) Straightforward calculation.

2) Let $A \stackrel{f}{\rightarrow} B \stackrel{g}{\rightarrow} C$ be a composition of  $\tau$-homomorphisms. Consider the diagram $Con(A) \stackrel{f^*}{\leftarrow} Con(B) \stackrel{g^*}{\leftarrow} Con(C)$. If $\theta \in Con(C)$ then, $f^*g^*(\theta) =  \{(x,y) \in A^2: (f(a),f(b)) \in g^*(\theta) \} = \{(x,y) \in A^2: (gf(a),gf(b)) \in \theta \} = (gf)^*(\theta)$. Hence $(gf)^* = f^*g^*$. 

3) Immediate

4) Let us assume that $f$ is a $\tau$-isomorphism. Then $f_*$ defines a bijective function $f_*:Con(A) \rightarrow Con(B)$. We first prove that $f^* f_* = 1_{Con(A)}$. Let $\theta \in Con(A)$. Then $(x,y) \in f^* f_*(\theta)$ iff $(f(x), f(y)) \in f_*(\theta)$ iff $(x,y) \in \theta$. Therefore  $f^* f_* = 1_{Con(A)}$. Now we prove that $f_* f^* = 1_{Con(B)}$. Let $\theta \in Con(B)$. Then $(x,y) \in f_* f^*(\theta)$ iff there exists $(x_0, y_0) \in f^*(\theta)$ such that $f(x_0) = x$ and $f(y_0) = y$. Since $(x_0, y_0) \in f^*(\theta)$ iff $(x,y) = (f(x_0), f(y_0)) \in \theta$, then we have that $f_* f^* = 1_{Con(B)}$. Thus $f_* = (f^*)^{-1}$. 

Let $f^{-1}$ be the inverse of $f$ and $\theta \in Con(A)$. Then $(x,y) \in (f^{-1})^*(\theta) \subseteq B^2$ iff $(f^{-1}(x), f^{-1}(y)) \in \theta$ iff $(ff^{-1}(x), ff^{-1}(y)) \in f_*(\theta)$ iff $(x, y) \in f_*(\theta)$. It proves that  $f_* = (f^*)^{-1} = (f^{-1})^*$. 

Now we prove that $f_*$ is an order preserving function. Suppose that $\theta_1 \subseteq \theta_2$ in $Con(A)$. Let $(c,d) \in f_*(\theta_1)$. Then $(f^{-1}(c), f^{-1}(d)) \in \theta_1 \subseteq \theta_2$ and $(c,d) \in f_*(\theta_2)$. Hence $f_*(\theta_1) \subseteq f_*(\theta_2)$ and $f_*$ is an order isomorphism from $Con(A)$ onto $Con(B)$. 

We first prove that $f'$ is well defined. If $x_\theta = y_\theta$ then, $(x,y)\in \theta$, $(f(x),f(y))\in f_*(\theta)$ and $f'(x_\theta) = f(x)_{f_*(\theta)} = f(y)_{/f_*(\theta)} =  f'(y_\theta)$. Thus $f'$ is well defined. If $f'(x_\theta) = f'(y_\theta)$ then, $(f(x), f(y)) \in f_*(\theta)$ and $(x,y) \in \theta$. Thus, $x_\theta = y_\theta$ and  $f'$ is injective.
We shall now proceed to prove that $f'$ is surjective. Let $y_{f_*(\theta)} \in B/f_*(\theta)$. Since $f$ is surjective, there exists $x\in A$ such that $f(x)= y$. Thus $y_{f_*(\theta)} = f(x)_{f_*(\theta)} = f'(x_\theta)$. Therefore $f'$ is surjective. Let $t(x_1 \ldots x_n) \in Term_\tau(X)$. Then for $a_1 \ldots a_n \in A$ we have that:  
\begin{eqnarray*}
f'(t^{A/\theta}({a_1}_\theta, \ldots, {a_n}_\theta)) & = &  f'(t^A(a_1,\ldots , a_n )_\theta) = f(t^A(a_1 ,\ldots , a_n ))_{f_*(\theta)}\\
& = & t^B(f(a_1) ,\ldots , f(a_n))_{f_*(\theta)} \\
 & = & t^{B/f_*(\theta)}(f(a_1)_{f_*(\theta)} ,\ldots ,f(a_n)_{f_*(\theta)})\\
 & = &  t^{B/f_*(\theta)}(f'({a_1}_\theta) , \ldots , f'({a_n}_\theta)).
\end{eqnarray*}
It proves that $f'$ preserve $\tau$-operations. Hence, $f'$ is a $\tau$-isomorphism.

5) Let us assume that $\theta_1, \theta _2 \in Con(A)$  are permutable. Since $f$ is a $\tau$-isomorphism each pair in $f_*(\theta_1) \circ f_*(\theta_2)$ has the form $(f(x), f(y))$ where $x,y \in A$. Suppose that $(f(x), f(y)) \in f_*(\theta_1) \circ f_*(\theta_2)$. Then, by definition of relational product, there exists $w\in A$ such that $(f(x), f(w)) \in f(\theta_1)$ and $(f(w), f(y)) \in f(\theta_2)$. Thus $(x,w) \in \theta_1$, $(w,y) \in \theta_2$ and $(x,y) \in \theta_1 \circ \theta_2 = \theta_2 \circ \theta_1$. It implies that there exists $v \in A$ such that $(x,v) \in \theta_2$ and $(v,x) \in \theta_1$ and consequently $(f(x), f(v)) \in f_*(\theta_2)$ and $(f(v),f(x)) \in f_*(\theta_1)$. Therefore $(f(x), f(y)) \in f_*(\theta_2) \circ f_*(\theta_1)$proving that $f_*(\theta_1)$, $f_*(\theta_2)$ are permutable.

\qed
\end{proof}

\begin{prop}\label{AUX1}
Let $A$ be an algebra, $\sigma \in Con(A)$ and the order isomorphism $u_\sigma: [\sigma, \nabla_A ] \rightarrow Con(A/\sigma)$ given by $u(\theta)=\theta/\sigma$. 
If $p:A \rightarrow A/\sigma$ is the natural homomorphism then $p^* = u_\sigma^{-1}$.
\end{prop}

\begin{proof}
Let $\theta \in [\sigma, \nabla]$. Then, by Eq.(\ref{QUOTCONG}), we have that 
\begin{eqnarray*}
p^*(\theta/\sigma) & = &  \{(x,y) \in A^2: (p(x), p(y)) \in \theta/\sigma \} = \{(x,y) \in A^2: (x_\sigma, y_\sigma) \in \theta/\sigma \}\\
 & = & \{(x,y) \in A^2: (x, y) \in \theta \}= \theta = u^{-1}(\theta/\sigma).
\end{eqnarray*}
Hence our claim.

\qed
\end{proof}

\begin{definition}\label{congoperator}
{\rm
Let ${\cal A}$ be a category of algebras. A {\it congruences operator} over ${\cal A}$ is a class operator of the form ${\cal A} \ni A \mapsto {\cal K}(A) \subseteq Con(A)$ such that,
\begin{enumerate}
\item
$\Delta_A \in {\cal K}(A) $.

\item
For each $\sigma \in {\cal K}(A)$, $A/\sigma \in {\cal A}$.

\item
If $f:A \rightarrow B$ is a ${\cal A}$-isomorphism then $f^* \upharpoonright_{{\cal K}(B)}: {\cal K}(B) \rightarrow {\cal K}(A)$ is an order isomorphism.

\end{enumerate}
}
\end{definition}

\begin{prop}\label{PRESHEAF}
Let ${\cal A}$ be a category of algebras and ${\cal K}$  be a congruences operator over ${\cal A}$. Let us define the class 
\begin{equation}\label{HOMAK}
Hom_{{\cal A}_{\cal K}}= \{A \stackrel{f}{\rightarrow} B \in Hom_{\cal A}: \exists \sigma \in {\cal K}(A)\hspace{0.1cm} s.t. \hspace{0.1cm} B \cong_{\cal A} A/\sigma  \}.
\end{equation}
Then the following assertions are equivalent:

\begin{enumerate}
\item
The pair ${\cal A}_{{\cal K}} = \langle Ob({\cal A}), Hom_{{\cal A}_{\cal K}}\rangle$ is a category and, by defining ${\cal K}(f) = f^* \upharpoonright_{{\cal K}(B)}$ for each $A \stackrel{f}{\rightarrow} B \in Hom_{{\cal A}_{\cal K}}$,  ${\cal K}: {\cal A}_{\cal K} \rightarrow Set$ is a presheaf.

\item
For each $A \in {\cal A}$ and $\sigma \in {\cal K}(A)$, if $p:A \rightarrow A/\sigma$ is the natural ${\cal A}$-homomorphism then the restriction $p^*\upharpoonright_{{\cal K}(A/\sigma)}$ is an order isomorphism from ${\cal K}(A/\sigma)$ onto ${\cal K}(A)\cap [\sigma, \nabla_A]$. 

\item 
$\theta \in {\cal K}(A)\cap[\sigma, \nabla_A ]$ iff $ \theta / \sigma \in {\cal K}(A/\sigma)$, for all $A \in {\cal A}$ and $\sigma \in {\cal K}(A)$.

\end{enumerate}
\end{prop}

\begin{proof}
$1\Longrightarrow 2$) Let us suppose that ${\cal A}_{{\cal K}}$ is a category and ${\cal K}: {\cal A}_{\cal K} \rightarrow Set$ is a presheaf.
Let $A \in {\cal A}$, $\sigma \in Con(A)$ and $p: A \rightarrow A/\sigma$ be the natural ${\cal A}$-homomorphism. Note that, $Imag(p^*\upharpoonright_{{\cal K}(A/\sigma)}) = Imag({\cal K}(p)) \subseteq {\cal K}(A)$ because ${\cal K}$ is a presheaf. Then, by Proposition \ref{AUX1}, $p^*\upharpoonright_{{\cal K}(A/\sigma)}$ is an injective order homomorphism of the form $p^*\upharpoonright_{{\cal K}(A/\sigma)} :{\cal K}(A/\sigma) \rightarrow {\cal K}(A)\cap[\sigma, \nabla_A]$. We want to prove that $p^*\upharpoonright_{{\cal K}(A/\sigma)}$ is a surjective map. For this we need to show that if $\theta \in {\cal K}(A)\cap[\sigma, \nabla_A]$ then $\theta/\sigma \in {\cal K}(A/\sigma)$. Indeed: By Theorem \ref{CORRESP}-1, $A/\theta \cong_{\cal A}(A/\sigma)/(\theta/\sigma)$ and therefore, the natural ${\cal A}$-homomorphism $A/\sigma \rightarrow (A/\sigma)/(\theta/\sigma)$ can be identify to the  ${\cal A}_{\cal K}$-homomorphism $g:A/\sigma \rightarrow A/\theta$ such that $g(x_{\sigma}) = x_{\theta}$. By hypothesis we have that ${\cal K}(g) = g^*\upharpoonright_{{\cal K}(A/\theta)}: {\cal K}(A/\theta) \rightarrow {\cal K}(A/\sigma)$ and  $\Delta_{A/\theta} \in {\cal K}(A/\theta)$. Then
\begin{eqnarray*}
{\cal K}(A/\sigma) \ni g^*(\Delta_{A/\theta}) & = & g^*(\theta /\theta) = \{(x_\sigma , y_\sigma) \in (A/\sigma)^2: (g(x_\sigma), g(y_\sigma)) \in \theta /\theta \} \\
  & = & \{(x_\sigma , y_\sigma) \in (A/\sigma)^2: (x_\theta, y_\theta) \in \theta /\theta \} \\
 & = & \{(x_\sigma , y_\sigma) \in (A/\sigma)^2: (x, y) \in \theta \}\\
 & = & \theta/\sigma
\end{eqnarray*}
proving that $\theta/\sigma \in  {\cal K}(A/\sigma)$. Thus, by Proposition \ref{AUX1}, if $\theta \in {\cal K}(A)\cap[\sigma, \nabla_A]$ then, $\theta/\sigma \in  {\cal K}(A/\theta)$ and  $[{\cal K}(p)](\theta/\sigma) = p^*(\theta/\sigma) = \theta$ proving that $ p^*\upharpoonright_{{\cal K}(A/\sigma)}$ is surjective. Hence our claim.

$2\Longrightarrow 3$) Immediate form Proposition \ref{AUX1}.

$3\Longrightarrow 1$) We first note that for each $A \in {\cal A}$, $1_A \in Hom_{{\cal A}_{\cal K}}$ because $\Delta_A \in {\cal K}(A)$. Now we prove that the class $Hom_{{\cal A}_{\cal K}}$ is closed by compositions. Let $A \in {\cal K}$, $\sigma \in {\cal K}(A)$, $\theta/\sigma \in {\cal K}(A/\sigma)$ and let us consider the following diagram  $A \stackrel{p_1}{\rightarrow} A/\sigma \stackrel{p_2}{\rightarrow} (A/\sigma)/(\theta/\sigma)$ in $Hom_{{\cal A}_{\cal K}}$ where $p_1$ and $p_2$ are two natural ${\cal A}$-homomorphisms. By Theorem \ref{CORRESP}-1, $(A/\sigma)/(\theta/\sigma) \cong_{\cal A} A/\theta$ and, by hypothesis, $\theta \in {\cal K}(A)$. Then the composition $p_2 p_1 \in Hom_{{\cal A}_{\cal K}}$ proving that $Hom_{{\cal A}_{\cal K}}$ is closed by compositions. Hence ${\cal A}_{\cal K}$ is a category. Now we prove that  ${\cal K}: {\cal A}_{\cal K} \rightarrow Set$ is a presheaf.  Let $f:A \rightarrow B \in Hom_{{\cal A}_{\cal K}}$. We first show that ${\cal K}(f) = f^* \upharpoonright_{{\cal K}(B)}$ is a function of the form ${\cal K}(f) = {\cal K}(B) \rightarrow {\cal K}(A)$. Note that $f$ admits the following factorization in ${\cal A}$
\begin{center}
\unitlength=1mm
\begin{picture}(20,20)(0,0)
\put(8,16){\vector(3,0){5}} \put(2,10){\vector(0,-2){5}}
\put(10,4){\vector(1,1){7}}

\put(2,10){\makebox(13,0){$\equiv$}}

\put(2,16){\makebox(0,0){$A$}} \put(20,16){\makebox(0,0){$B$}}
\put(2,0){\makebox(0,0){$A/\sigma$}}
\put(2,20){\makebox(17,0){$f$}} \put(2,8){\makebox(-6,0){$p$}}
\put(18,2){\makebox(-4,3){$g$}}
\end{picture}
\end{center}
where $\sigma \in {\cal K}(A)$, $p$ is the natural ${\cal A}$-homomorphism and $g$ is an ${\cal A}$-isomorphism. By hypothesis and Theorem \ref{CORRESP}, $p^*: {\cal K}(A/\sigma) \rightarrow {\cal K}(A) \cap [\sigma, \nabla_A]$ is an order isomorphism and $g^*\upharpoonright_{{\cal K}(B)}: {\cal K}(B) \rightarrow {\cal K}(A/\sigma)$ is an order isomorphism because $g$ is a ${\cal A}$-isomorphism. Thus, by Proposition \ref{AUX}-2, 
$f^* = (pg)^* = p^*g^*$ and then, $f^*\upharpoonright_{{\cal K}(B)}$ is an order homomorphism from ${\cal K}(B)$ onto ${\cal K}(A)$. By Proposition \ref{AUX}-(2 and 3) we also note that ${\cal K}(-)$ is a contravariant functor. Hence  ${\cal K}: {\cal A}_{\cal K} \rightarrow Set$ is a presheaf.

\qed
\end{proof}

\begin{definition}
{\rm
Let ${\cal A}$ be a category of algebras. A congruences operator ${\cal K}$ over ${\cal A}$ satisfying the equivalent conditions of Proposition \ref{PRESHEAF} is called a {\it congruences presheaf}.   
}
\end{definition}

\begin{example}\label{VARIETYPREC} 
{\rm [{\it Presheaf $Con$}]
Let ${\cal A}$ be a category of algebras closed under homomorphic images. Let us define the class operator ${\cal A} \ni A \mapsto Con(A)$. It is not difficult to show that  $Con(-)$ is a congruences operator and that $Hom_{{\cal A}_{Con}}$ is the class of surjective ${\cal A}$-homomorphisms. Thus ${\cal A}_{Con}$ is a category. If we define $Con(f) = f^*$  then, by Proposition \ref{AUX}, $Con$ is a congruences presheaf. In particular $Con(-)$ is a congruences presheaf over varieties of algebras.      
}
\end{example}

\begin{example}\label{QUASIVARIETYPREC}
{\rm Let ${\cal A}$ be a quasivariety. For each $A\in {\cal A}$, let us consider the set of relative congruences of $A$, $Rel(A) = \{\theta \in Con(A): A/\theta \in {\cal A} \}$. Let us 
define the class operator ${\cal A}\ni A \mapsto Rel(A)$.  It is not difficult to prove that $Rel(-)$ is a congruences operator and that ${\cal A}_{Rel} = \langle Ob({\cal A}), Hom_{{\cal A}_{Rel}} \rangle$ is a category. We shall prove that if  $f:A \rightarrow B \in Hom_{{\cal A}_{Rel}}$ then, $Imag(f^*) \subseteq Rel(A)$ or equivalently, for each $\theta \in Rel(B)$, $A/f^*(\theta) \in {\cal A}$. Indeed: Let us consider a quasi equation $(\&_{i=1}^{n} r_i(\overline{x}) = s_i(\overline{x})) \Longrightarrow r(\overline{x}) = s(\overline{x}) $  holding in ${\cal A}$ where $\overline{x}$ is a vector of $k$ variables. Let $\overline{a}_{f^*(\theta)}$ be a vector of $k$ elements of the algebra $A/f^*(\theta)$ such that $A/f^*(\theta) \models \&_{i=1}^{n} r_i(\overline{a}_{f^*(\theta)}) = s_i(\overline{a}_{f^*(\theta)})$. Thus, by definition of $f^*$ in Eq.(\ref{UPF}), we have that $(f(s_i(\overline{a}),f(r_i(\overline{a})) = (s_i(f(\overline{a})),r_i(f(\overline{a})) \in \theta$ for $1\leq i \leq n$ and then, $B/\theta \models \&_{i=1}^{n} r_i(f(\overline{a}))_{\theta} = s_i(f(\overline{a}))_{\theta}$. Since $B/\theta \in {\cal A}$ and the quasi equation holds in ${\cal A}$,  $B/\theta \models s(f(\overline{a}))_{\theta} = r(f(\overline{a}))_{\theta}$. It implies that $(f(s(\overline{a})), f(r(\overline{a}))) \in \theta$ and then $(s(\overline{a}), r(\overline{a})) \in f^*(\theta)$. Hence $A/f^*(\theta) \models r(\overline{a}_{f^*(\theta)}) = s(\overline{a}_{f^*(\theta)})$. It proves that $A/f^*(\theta) \in {\cal A}$ and, by Proposition \ref{AUX}, $Rel(-)$ is a congruences presheaf.
}
\end{example}

\begin{prop}\label{CORRESP2}
Let ${\cal A}$ be a category of algebras, ${\cal K}$ be a congruences presheaf and $A \in {\cal A}$. If $\sigma \in {\cal K}(A)$, $\theta \in {\cal K}(A)\cap [\sigma, \nabla_A]$ and $A \cong_{{\cal A}} A/\theta$ then there exists $\theta' \in {\cal K}(A/\sigma)$ such that $A \cong (A/\sigma)/\theta'$.

\end{prop}

\begin{proof}
Since $\theta \in {\cal K}(A)\cap [\sigma, \nabla_A]$, by Proposition \ref{PRESHEAF}-3, $\theta'= \theta/\sigma \in {\cal K}(A/\sigma)$. Then, by Theorem \ref{CORRESP}-1, $(A/\sigma)/\theta' = (A/\sigma)/(\theta/\sigma) \cong_{{\cal A}} A/\theta \cong A$.  

\qed
\end{proof}

\begin{definition}
{\rm Let ${\cal A}$ be a category of algebras and ${\cal K}$ be a congruences presheaf. An algebra $A \in {\cal A}$ has the {\it Cantor-Bernstein-Schr\"oder property} with respect to  ${\cal K}$ ({\it $CBS_{{\cal K}}$-property} for short)  iff the following holds: Given $B \in {\cal A}$ and $\theta_B \in {\cal K}(B)$ such that there is $\theta_A \in {\cal K}(A)$ with $A\cong_{\cal A} B/\theta_B$ and $B\cong_{\cal A} A/\theta_A$ it follows that $A \cong_{\cal A} B$.
}
\end{definition}

\begin{theo}\label{CBS1}
Let ${\cal A}$ be a category of algebras and ${\cal K}$ be a congruences presheaf. Then, the following conditions are equivalent for each $A\in {\cal A}$:

\begin{enumerate}
\item
$A$ has the $CBS_{{\cal K}}$-property.

\item
If $\theta \in {\cal K}(A)$ and $A \cong_{\cal A} A/\theta$ then, for all $\sigma \in {\cal K}(A)$ such that $\sigma \subseteq \theta$ we have that $A \cong_{\cal A} A/\sigma$.

\end{enumerate}
\end{theo}

\begin{proof}
{\rm $1 \Longrightarrow 2$)  Let $\sigma , \theta \in {\cal K}(A)$ such that $\sigma \subseteq \theta$ and $A \cong_{\cal A} A/\theta$. Let $B = A/\sigma$. 
Note that $\theta \in {\cal K}(A) \cap [\sigma, \nabla_A]$ then, by Proposition \ref{CORRESP2}, there exists $\theta_B \in {\cal K}(A/\sigma)= {\cal K}(B) $ such that $A \cong_{\cal A} B/\theta_B$. Since $A$ has the $CBS_{{\cal K}}$-property we have that $A \cong_{\cal A} B = A/\sigma$.       

$2 \Longrightarrow 1$) 
Let $B$ in ${\cal A}$, $\sigma_A \in {\cal K}(A)$, $\sigma_B \in {\cal K}(B)$. Suppose that there exists two ${\cal A}$-isomorphisms $f: A\rightarrow B/\sigma_B$ and  $g: B\rightarrow A/\sigma_A$. 

By Proposition \ref{AUX}-4, $g_*(\sigma_B) \in {\cal K}(A/\sigma_A)$ and there exists a ${\cal A}$-isomorphism $g': B/\sigma_B \rightarrow (A/\sigma_A)/g_*(\sigma_B)$. Let us consider the following composition of ${\cal A}$-isomorphisms: 
\begin{equation}\label{COMPCBS1}
A \stackrel{f}{\rightarrow} B/\sigma_B \stackrel{g'}{\rightarrow} (A/\sigma_A)/g_*(\sigma_B).
\end{equation}
Note that $g_*(\sigma_B) = \theta/\sigma_A$ for some $\theta \in Con(A)$ and, by Proposition \ref{PRESHEAF}-3, $\theta \in {\cal K}(A) \cap[\sigma_A, \nabla_A]$. Thus, by Theorem \ref{CORRESP}-1, we have that $(A/\sigma_A)/g_*(\sigma_B) = (A/\sigma_A)/(\theta/\sigma_A) \cong_{\cal A} A/\theta$ and the diagram of ${\cal A}$-isomorphisms given in Eq.(\ref{COMPCBS1}) can be seen as  
\begin{equation}\label{COMPCBS2}
A \stackrel{f}{\rightarrow} B/\sigma_B \stackrel{g'}{\rightarrow} A/\theta.
\end{equation}
Therefore $A \cong_{\cal A} A/\theta$ where $\theta \in {\cal K}(A) \cap[\sigma_A, \nabla_A]$. Since $\sigma_A \subseteq \theta$, by hypothesis, $A \cong_{\cal A} A/\sigma_{A} \cong_{\cal A} B$. Hence $A$ has the $CBS_{{\cal K}}$-property.  
}

\qed
\end{proof}

\begin{rema}\label{NONTRIVIALCBS}
{\rm Let us notice that, by the condition 2 of the Theorem \ref{CBS1}, if there would be no $\theta \in {\cal K}(A)$ such that $A \cong_{\cal A} A/\theta$ then the algebra $A$ trivially has the $CBS_{{\cal K}}$-property. Then we say that {\it $A$ satisfies the $CBS_{{\cal K}}$-property in a non trivial way} whenever this property is satisfied and there exists $\theta \in {\cal K}(A)$ such that $A \cong_{\cal A} A/\theta$.} 
\end{rema}

We conclude this section with an example illustrating our abstract framework for the CBS-theorem.

\begin{example}\label{PSEUDOSIMPLE}
{\rm [Pseudo-simple algebras] An algebra $A$ is called {\it pseudo-simple} \cite{MONK} iff $Card(A)> 1$ and for every $\sigma \in Con(A) - \{\nabla_A\}$, $A/\sigma \cong A$. Let ${\cal A}$ be category of algebras closed under homomorphic images and let us consider the congruences presheaf $Con(-)$ (see Example \ref{VARIETYPREC}). Then, by Theorem \ref{CBS1}, pseudo-simple algebras of ${\cal A}$ satisfy the $CBS_{Con}$-property.

Concrete examples of these algebras, can be found in the variety ${\cal G}rp$ of groups. Indeed, a {\it quasi-cyclic group} is an Abelian group which is isomorphic to $Z(p^\infty)$ for some prime number $p$. They are pseudo-simple algebras in ${\cal G}rp$. In this way quasi-cyclic groups have the $CBS_{Con}$-property. 
}
\end{example}

\section{Factor congruences presheaves} \label{FCPRES}

In this section we introduce and study a special case of congruences presheaf related to factor congruences. In this framework necessary and sufficient conditions for the validity of
CBS-Theorem are established.

\begin{definition}\label{FACTORCPRESHEAF}
{\rm
Let ${\cal A}$ be a category of algebras. A {\it factor congruences presheaf} is a congruences presheaf ${\cal K}$ such that for each  $A\in {\cal A}$,

\begin{enumerate}

\item
${\cal K}(A) \subseteq FC(A)$.

\item
For each $\theta \in {\cal K}(A)$ there exists $\neg \theta \in {\cal K}(A)$, such that $(\theta, \neg \theta)$ is a pair of factor congruences on $A$.

\item
If $\sigma \in {\cal K}(A)$, $\theta \in {\cal K}(A) \cap [\sigma, \nabla_A]$ and $(\theta, \neg \theta)$ is a pair of factor congruences in ${\cal K}(A)$ then 
$(\theta/\sigma, (\neg \theta \lor \sigma)/\sigma)$ is a pair of factor congruences in ${\cal K}(A/\sigma)$.  

\end{enumerate}
}
\end{definition}

By item 2 in the above definition, $\nabla_A \in {\cal K}(A)$ because $\Delta_A \in {\cal K}(A)$ and, by Proposition \ref{PRESHEAF}, the following result is immediate.

\begin{prop}\label{AUX4}
Let ${\cal A}$ be a category of algebras and ${\cal K}$ be a factor congruences presheaf. Let $A\in {\cal A}$, $\sigma \in {\cal K}(A)$, $\theta \in {\cal K}(A) \cap [\sigma, \nabla_A]$ and $(\theta, \neg \theta)$ be a pair of factor congruences in ${\cal K}(A)$. Then $\neg \theta \lor \sigma \in {\cal K}(A) \cap [\sigma, \nabla_A]$. 

\qed
\end{prop}

Let ${\cal A}$ be a category of algebras such that for each $A \in {\cal A}$ and $\sigma \in FC(A)$, $A/\sigma \in {\cal A}$. The, by Proposition \ref{AUX}-5, it is immediate that the class operator 
\begin{equation}\label{CLASOPFACT}
{\cal A} \ni A \mapsto FC(A)
\end{equation}
is a congruence operator. The following proposition provides a sufficient condition for $FC(-)$ to be a congruences presheaf.

\begin{prop}\label{MODPERM}
Let ${\cal A}$ be a category of algebras such that for each $A \in {\cal A}$ and $\sigma \in FC(A)$, $A/\sigma \in {\cal A}$. If ${\cal A}$ is congruence modular or congruence permutable then $FC(-)$ is a congruences presheaf.
\end{prop}

\begin{proof}
Let us assume that ${\cal A}$ is congruence modular. Let $A \in {\cal A} $, $\sigma \in FC(A)$, $\theta \in FC(A) \cap [\sigma, \nabla_A]$ and $(\theta, \neg \theta)$ be a pair of factor congruences in $FC(A)$. 

We first prove that $(\theta /\sigma, \neg \theta \lor \sigma /\sigma)$ is a pair of factor congruences in $FC(A/\sigma)$.  By the modularity $\theta \cap (\sigma \lor \neg \theta) = \sigma \lor (\theta \cap \neg \theta) = \sigma \lor \Delta_A = \sigma$ because $\sigma \subseteq \theta$. Then, by Theorem \ref{CORRESP}-2, $\theta/\sigma \cap (\neg \theta \lor \sigma)/\sigma = \Delta_{A/\sigma}$. We also note that $\nabla_A = \theta \circ \neg \theta \subseteq  \theta \circ (\neg \theta \lor \sigma)$. Then, by Theorem \ref{CORRESP}-3, $\theta/\sigma \circ (\neg \theta \lor \sigma)/\sigma = \nabla_{A/\sigma}$. Thus, $(\theta/\sigma, (\neg \theta \lor \sigma)/\sigma )$ is a pair of factor congruences on $A/\sigma$ and $\theta/\sigma \in FC(A/\sigma)$. 
Now if we suppose  that $\theta/\sigma \in FC(A/\sigma)$ then, by Proposition \ref{BFCQUOT}, $\theta \in FC(A)$. Hence, by Proposition \ref{PRESHEAF}, $FC(-)$ is a factor congruences presheaf.  
Let us notice that if ${\cal A}$ is a category of congruence permutable algebras then, by the Birkhoff theorem (see \cite[Proposition 5.10]{Bur}), ${\cal A}$ is congruence modular. Hence or claim.

\qed
\end{proof}

\begin{example}\label{INJMODDIV}
{\rm [$CBS_{FC}$-property: injective modules and divisible groups] Let ${\cal M}od_R$ be the variety of modules over the ring $R$ and ${\cal A}b$ be the variety of Abelian groups. Let us notice that divisible groups are the injective objects in ${\cal A}b$.  We will indistinctly denote the varieties ${\cal M}od_R$ and ${\cal A}b$ with ${\cal A}$. In the variety ${\cal A}$, the notion of finite direct sum and finite direct product coincides. Thus, for each $A \in {\cal A}$, $\langle FC(A), \subseteq \rangle $ is order reverse isomorphic to the set of direct factor sub algebras of $A$ denoted by  $\langle DF(A), \subseteq \rangle $. It is well known that ${\cal A}$ is a congruence permutable variable and then, by Proposition \ref{MODPERM},  $FC(-)$ is a congruences presheaf.

Let $A$ be an injective object in ${\cal A}$. We shall prove that $A$ has the $CBS_{FC}$-property. In order to do this, by Theorem \ref{CBS1}, we have to show that: for $I, K \in DF(A)$ such that $I$ is a sub algebra of $K$, if $A \cong_{\cal A} I$ then $A \cong_{\cal A} K$. 

Indeed: Let $f:I \rightarrow A$ be a ${\cal A}$-isomorphism. Since $A$ is injective, there exists ${\cal A}$-homomoprhism $g:K \rightarrow A$ such that the following diagram commute       
    
\begin{center}
\unitlength=1mm
\begin{picture}(20,20)(0,0)
\put(8,16){\vector(3,0){5}} \put(2,10){\vector(0,-2){5}}
\put(10,4){\vector(1,1){7}}

\put(2,10){\makebox(13,0){$\equiv$}}

\put(2,16){\makebox(0,0){$I$}} \put(20,16){\makebox(0,0){$A$}}
\put(2,0){\makebox(0,0){$K$}}
\put(2,20){\makebox(17,0){$f$}} \put(2,8){\makebox(-6,0){$1_I$}}
\put(18,2){\makebox(-4,3){$g$}}
\end{picture}
\end{center}
Let us notice that  $g 1_I$ is an injective ${\cal A}$-homomoprhism. Thus, if we consider the following composition $K \stackrel{f^{-1}\upharpoonright_{K}}{\rightarrowtail} I \stackrel{g1_I }{\rightarrowtail} A$, by the commutativity of the above diagram, we have that $ A \supseteq K \ni x = f (f^{-1}(x)) = g 1_I (f^{-1}(x)) $  proving that the diagram $K \stackrel{f^{-1}\upharpoonright_{K}}{\rightarrowtail} I \stackrel{g1_I }{\rightarrowtail} K$ is the identity $1_K$. It implies that $g 1_I $ is also a surjective ${\cal A}$-homomoprhism and $I \cong_{\cal A} K $. Hence $A \cong_{\cal A} K $ and $A$ has the $CBS_{FC}$-property. Since $A$ is an injective object then the denumerable direct product $B = \prod_{{\mathbb N}}A$ is injective in ${\cal A}$. Thus, by Proposition \ref{NDIRECTPRODUCT}, there exists $\sigma \in FC(B)$ such that $B \cong_{\cal A} B/\sigma$. In this way $B$ satisfies the $CBS_{FC}$-property in a non trivial way.
}
\end{example}

Now we study a necessary a sufficient condition for the validity of the $CBS$-property respect to a factor congruences presheaf. \\

Let ${\cal A}$ be a category of algebras and ${\cal K}$ be a factor congruences presheaf.  Let $A \in {\cal A}$, $\theta \in {\cal K}(A)$ and let us suppose that there exists an ${\cal A}$-isomorphism  $f:A\rightarrow A/\theta$. By Theorem \ref{CORRESP}-2 and Proposition \ref{AUX}-4 let us consider the $\langle \nabla, \Delta, \subseteq \rangle$-isomorphism $\hat{f} = u_\theta^{-1}f_*$ i.e.,  
\begin{equation}\label{fWIDE}
\hat{f}:  {\cal K}(A)  \stackrel{f_*}{\rightarrow} {\cal K}(A/\theta) \stackrel{u_\theta^{-1}}{\rightarrow} {\cal K}(A)\cap[\theta, \nabla_A].
\end{equation}
If  $\sigma \in {\cal K}(A)$ such that $\sigma \subseteq \theta$ then we define the following set: 
\begin{equation}\label{SIGMAIS}
\langle \sigma \rangle_{\theta}= \{\zeta \in [\Delta_A,\theta]\cap {\cal K}(A): A/\sigma \cong_{\cal A} A/\zeta\}.
\end{equation}
If $\zeta \in \langle \sigma \rangle_{\theta}$ then we recursively define the following sequences of congruences: 
\begin{align}\label{ALIN1}
\sigma_0 &= \Delta_A \,,&   \nonumber \\     
\sigma_1 &= \zeta \,,&                 \theta_1 &= f_*(\sigma_0) =  \theta/\theta \,,\ \nonumber\\    
\sigma_2 &= u_\theta^{-1}(\theta_1) = \theta \,,& \theta_2 &= f_*(\sigma_1) \,,\ \nonumber \\
\sigma_3 &= u_\theta^{-1}(\theta_2) \,,&          \theta_3 &= f_*(\sigma_2) \,,\ \nonumber \\
         &\vdots                             \,&                   &\vdots               \,\ \nonumber \\
\sigma_{n+1} &= u_\theta^{-1}(\theta_n) \,,&      \theta_{n+1} &= f_*(\sigma_n) \,.\
\end{align}
By Eq.(\ref{fWIDE}), $(\sigma_n)_{n\in {\mathbb{N}}}$ is a sequence in ${\cal K}(A)$.

\begin{prop}\label{INCREASINGSIGMAN}
Let ${\cal A}$ be a category of algebras, ${\cal K}$ be a factor congruences presheaf, $A \in {\cal A}$ and $\theta \in {\cal K}(A)$ such that 
there exists an ${\cal A}$-isomorphism  $f:A\rightarrow A/\theta$. Let us consider the sequence $(\sigma_n)_{n\in {\mathbb{N}}}$ in ${\cal K}(A)$
given in {\rm Eq.(\ref{ALIN1})}. Then: 
\begin{enumerate}

\item
$\hat{f}(\sigma_n) = \sigma_{n+2}$,

\item
$(\sigma_n)_{n \in {\mathbb{N}}} $ is an increasing sequence in ${\cal K}(A)$. In particular, if $\Delta_A < \zeta$ then $(\sigma_n)_{n \in {\mathbb{N}}} $ is strictly increasing.

\end{enumerate}
\end{prop}

\begin{proof}
1) If $k \geq 2$ then $\sigma_k = u_\theta^{-1}(\theta_{k-1}) = u_\theta^{-1}f_*(\sigma_{k-2}) = \hat{f}(\sigma_{k-2})$. Thus if $k = n+2$ we have that $\hat{f}(\sigma_n) = \sigma_{n+2}$.

2) Suppose that  $\sigma_0 = \Delta_A = \zeta = \sigma_1$. Then it is not very hard to see that $\sigma_n = \theta$ for $n\geq 2$. Thus $(\sigma_n)_{n \in {\mathbb{N}}} $ is an increasing sequence in ${\cal K}(A)$. Let us assume that  $\sigma_0 = \Delta_A < \zeta = \sigma_1$. By induction, let us assume that  $\sigma_{i} < \sigma_{j}$ whenever  $1 < i < j < n$. Since $\hat{f}$ is an order isomorphism and $n\geq 2$, by item 1, we have that $\sigma_n = \hat{f}(\sigma_{n-2}) < \hat{f}(\sigma_{n-1}) = \sigma_{n+1}$. Hence $(\sigma_n)_{n\in {\mathbb{N}}}$ is strictly increasing.

\qed
\end{proof}

\begin{definition}\label{CBSSEQURNCE}
{\rm Let ${\cal A}$ be a category of algebras, ${\cal K}$ be a factor congruences, $A \in {\cal A}$ and $\theta \in {\cal K}(A)$ such that 
there exists an ${\cal A}$-isomorphism  $f:A\rightarrow A/\theta$. Let us consider the sequence $(\sigma_n)_{n\in {\mathbb{N}}}$ in ${\cal K}(A)$ given in {\rm Eq.(\ref{ALIN1})}.  Then a {\it CBS-sequence} is a sequence of the form $(\sigma_{2n}\lor\neg \sigma_{2n+1})_{n\geq 0}$ such that  

\begin{enumerate}

\item
$\neg \sigma_1 =  \neg \zeta = \hat{f}^{-1}( \neg \hat{f}(\zeta))$ where $(\hat{f}(\zeta),\neg \hat{f}(\zeta))$ is a pair of factor congruences in  ${\cal K}(A)$. 

\item
$\neg \sigma_{2n+3} = \hat{f}(\neg \sigma_{2n+1}) $ for $n\geq 1$.

\end{enumerate}
}
\end{definition}

Let us note that $(\zeta, \neg \zeta )$ is a pair of factor congruences because $\hat{f}$ preserves order and permutability in view of Proposition \ref{AUX}-5.

\begin{prop}\label{INCREASINGSIGMAN2}
Let ${\cal A}$ be a category of algebras, ${\cal K}$ be a factor congruences presheaf, $A \in {\cal A}$ and $\theta \in {\cal K}(A)$ such that 
there exists an ${\cal A}$-isomorphism  $f:A\rightarrow A/\theta$. Let us consider the sequence $(\sigma_n)_{n\in {\mathbb{N}}}$ in ${\cal K}(A)$ given in {\rm Eq.(\ref{ALIN1})} and a CBS-sequence $(\sigma_{2n}\lor\neg \sigma_{2n+1})_{n\geq 0}$. Then:

\begin{enumerate}
\item
$\sigma_{2n + 1} \lor \neg \sigma_{2n + 1} = \nabla_A$.

\item
$(\sigma_{2n}\lor\neg \sigma_{2n+1})_{n\geq 0}$ is a dual orthogonal sequence in ${\cal K}(A)$. 

\item
$\hat{f}(\sigma_{2n}\lor\neg \sigma_{2n+1}) = \sigma_{2n+2}\lor\neg \sigma_{2n+3}$ for $n\geq 0$.

\end{enumerate}

\end{prop}

\begin{proof}
1) By definition of CBS-sequence, $(\sigma_1, \neg \sigma_1)$ is a pair of factor congruences in  ${\cal K}(A)$ and then $\sigma_1 \lor \neg \sigma_1= \nabla_A$. Since $\hat{f}$ is an order isomorphism, if $n>0$ and  $\sigma_{2(n-1) + 1} \lor \neg \sigma_{2(n-1) + 1} = \nabla_A$ then $\nabla_A = \hat{f}(\sigma_{2(n-1) + 1} \lor \neg \sigma_{2(n-1) + 1}) = \hat{f}(\sigma_{2(n-1) + 1}) \lor \hat{f}(\neg \sigma_{2(n-1) + 1}) = \sigma_{2(n-1) + 3} \lor \neg \sigma_{2(n-1) + 3} = \sigma_{2n + 1} \lor \neg \sigma_{2n + 1}$.  

2) By Proposition \ref{INCREASINGSIGMAN}-2, for each natural number $n$ we have that $\sigma_{2n} \leq \sigma_{2n+1}$ and then, $\sigma_{2n+1} \in  {\cal K}(A)\cap [\sigma_{2n}, \nabla_A]$. 
 Thus, by Definition \ref{FACTORCPRESHEAF}-3 and Proposition \ref{AUX4}, $ \sigma_{2n} \lor \neg \sigma_{2n+1} \in {\cal K}(A)$. In this way, $ (\sigma_0\lor\neg \sigma_1, \sigma_2\lor \neg \sigma_3, \ldots )=(\sigma_{2n}\lor\neg \sigma_{2n+1})_{n\geq 0}$ is a sequence in ${\cal K}(A)$.  
Suppose that $m < n$. Since $(\sigma_n)_{n\in {\mathbb{N}}}$ is an increasing sequence $\sigma_{2n} \geq \sigma_{2m + 1}$ then, by item 1, we have that,  
\begin{eqnarray*}
(\sigma_{2m}\lor\neg \sigma_{2m+1}) \lor(\sigma_{2n}\lor\neg \sigma_{2n+1}) & \geq & \sigma_{2m}\lor (\neg \sigma_{2m+1} \lor \sigma_{2m + 1} )\lor\neg \sigma_{2n+1} \\
 & = & \sigma_{2m}\lor \nabla_A \lor \neg \sigma_{2n+1} = \nabla_A.
\end{eqnarray*}
Hence $(\sigma_{2n}\lor\neg \sigma_{2n+1})_{n\geq 0}$ is a dual orthogonal sequence in ${\cal K}(A)$. 

4) Since $\hat{f}$ is a an order isomorphism, by Proposition \ref{INCREASINGSIGMAN}-1,  $\hat{f}(\sigma_{2n}\lor \neg \sigma_{2n+1}) = \hat{f}(\sigma_{2n})\lor \hat{f}(\neg \sigma_{2n+1}) = \sigma_{2n+2}\lor\neg \sigma_{2n+3}$.

\qed
\end{proof}

In what follows, the infimum in ${\cal K}(A)$ of a family $(\sigma_i)_{i\in I}$ of ${\cal K}(A)$, if it exists, will be denoted by  $\bigsqcap^{{\cal K}(A)}_{i\in I}\sigma_i$, to distinguish it from the infimum $\bigcap_{i\in I}\sigma_i$  in $Con(A)$, which need not belong to ${\cal K}(A)$.

\begin{definition}\label{CBSCOMPLETE}
{\rm
Let ${\cal A}$ be a category of algebras and ${\cal K}$ be a factor congruences presheaf. An algebra $A\in {\cal A}$ is called {\it  $CBS_{\cal K}$-complete} iff for all
${\cal A}$-isomorphism $f:A \rightarrow A/\theta$, where $\theta \in {\cal K}(A)$, and for all $\sigma \in {\cal K}(A)$ such that $\sigma \subseteq \theta$ there exists $\zeta \in \langle \sigma \rangle_{\theta}$ and a CBS-sequence $(\sigma_{2n}\lor\neg \sigma_{2n+1})_{n\geq 0}$ satisfying,
\begin{enumerate}
\item
$\sigma_{\zeta} = \bigsqcap^{{\cal K}(A)}_{n\geq 1}(\sigma_{2n}\lor\neg \sigma_{2n+1})$ exists. 

\item
There exists $\neg \sigma_{\zeta} \in {\cal K}(A)$ such that $(\sigma_{\zeta}, \neg \sigma_{\zeta})$  and $(\neg\zeta \cap \sigma_{\zeta}, \zeta \lor \neg \sigma_{\zeta})$ are two pairs of factor congruences in  ${\cal K}(A)$.

\end{enumerate}
}
\end{definition}

\begin{theo} \label{CBS2}
Let ${\cal A}$ be a category of algebras, ${\cal K}$ be a factor congruences presheaf  and $A \in {\cal A}$. Then the following conditions are equivalent: 
\begin{enumerate}
\item
$A$ is $CBS_{\cal K}$-complete. 

\item
$A$ has the $CBS_{\cal K}$-property.

\end{enumerate}
\end{theo}

\begin{proof}
$1 \Longrightarrow 2)$ Let us assume that $A$ is $CBS_{\cal K}$-complete. Let $\sigma, \theta \in {\cal K}(A)$ such that $\sigma \subseteq \theta$ and  $f:A \rightarrow A/\theta$ be a ${\cal A}$-isomorphism. By Theorem \ref{CBS1} we shall prove that $A\cong_{\cal A} A/\sigma $. Let us suppose that $(\sigma_{2n}\lor\neg \sigma_{2n+1})_{n\geq 0}$ is a CBS-sequence satisfying the condition of Definition \ref{CBSCOMPLETE}.

By hypothesis $\sigma_{\zeta} = \bigsqcap^{{\cal K}(A)}_{n\geq 1}(\sigma_{2n}\lor\neg \sigma_{2n+1}) \in  {\cal K}(A)\cap [\zeta, \nabla_A]$ and there exists $\neg \sigma_{\zeta} \in {\cal K}(A)$ such that $(\sigma_{\zeta}, \neg \sigma_{\zeta})$  and $(\neg\zeta \cap \sigma_{\zeta}, \zeta \lor \neg \sigma_{\zeta})$ are two pairs of factor congruences in  ${\cal K}(A)$. If we define $\chi = \neg \zeta \cap \sigma_{\zeta}$ and $\neg \chi = \neg \sigma_{\zeta} \lor \zeta $ then 
\begin{equation}\label{1}
 A\cong A/ \neg \chi \times A /\chi.  
\end{equation}
Since $\sigma_{\zeta} \in  {\cal K}(A)\cap [\zeta, \nabla_A]$, by Proposition \ref{AUX4} and by hypothesis, we have that
\begin{eqnarray}\label{2}
A/\zeta & \cong_{\cal A} & A/(\neg \sigma_{\zeta} \lor \zeta) \times A/\sigma_{\zeta} \nonumber \\
& = & A/\neg \chi \times A/\sigma_{\zeta}. 
\end{eqnarray}
Since $f_*(\chi) \in {\cal K}(A/\theta)$, by Theorem \ref{PRESHEAF}-3 there exists $\rho \in {\cal K}(A) \cap [\theta, \nabla_A]$ such that $f_*(\chi) = \rho/\theta$. Therefore  
$\hat{f}(\chi) = u_\theta^{-1} f_*(\chi) = u_\theta^{-1}(\rho/\theta) = \rho$ and, by Proposition \ref{AUX}-4, we have that
\begin{equation}\label{3}
A/\chi  \cong_{\cal A}  (A/\theta)/f_*(\chi) = (A/\theta)/(\rho/\theta ) \cong_{\cal A} A/\rho = A/\hat{f}(\chi).
\end{equation}

Since $\hat{f}: {\cal K}(A) \rightarrow {\cal K}(A) \cap[\theta, \nabla_A]$ is a $\langle \nabla, \Delta, \subseteq \rangle$-isomorphism, by Definition \ref{CBSSEQURNCE}, we have that
\begin{eqnarray}\label{4}
\hat{f}(\chi) &=& \hat{f}(\sigma_{\zeta} \cap \neg \zeta) \nonumber \\
& =& \hat{f}(\mbox{$\bigsqcap^{{\cal K}(A)}_{n\geq 1}$}(\sigma_{2n}\lor\neg \sigma_{2n+1})) \cap \hat{f}(\neg \zeta) \nonumber \\
& =& \hat{f}(\mbox{$\bigsqcap^{{\cal K}(A)}_{n\geq 1}$}(\sigma_{2n}\lor\neg \sigma_{2n+1})) \cap \hat{f}(\hat{f}^{-1}(\neg \hat{f}(\zeta)) \nonumber \\
& =& \hat{f}(\mbox{$\bigsqcap^{{\cal K}(A)}_{n\geq 1}$}(\sigma_{2n}\lor\neg \sigma_{2n+1})) \cap \hat{f}(\Delta_A \lor \hat{f}^{-1}(\neg \hat{f}(\zeta)) \nonumber \\
& =& \hat{f}(\mbox{$\bigsqcap^{{\cal K}(A)}_{n\geq 1}$}(\sigma_{2n}\lor\neg \sigma_{2n+1})) \cap ( \hat{f}(\Delta_A) \lor \hat{f}(\hat{f}^{-1}(\neg \hat{f}(\zeta)) ) \nonumber \\
& = & \mbox{$\bigsqcap^{{\cal K}(A)}_{n\geq 1}$}(\sigma_{2n+2}\lor\neg \sigma_{2n+3}) \cap (\theta \lor \neg \hat{f}(\zeta)) \nonumber \\
& = &  \mbox{$\bigsqcap^{{\cal K}(A)}_{n\geq 1}$}(\sigma_{2n+2}\lor\neg \sigma_{2n+3}) \cap (\sigma_2 \lor \neg \sigma_3) \nonumber \\
& = & \mbox{$\bigsqcap^{{\cal K}(A)}_{n\geq 1}$}(\sigma_{2n} \lor \sigma_{2n+1}) \nonumber \\ 
& = & \sigma_\zeta
\end{eqnarray}

From Eq.(\ref{3}) and Eq.(\ref{4}), $A/\chi \cong_{\cal A} A/\sigma_{\zeta}$. Then, by Eq.(\ref{2}), $A/\zeta \cong_{\cal A} A/\neg \chi \times A/\chi $ and, by equation Eq.(\ref{1}), $A \cong_{\cal A} A/\zeta \cong_{\cal A}  A/\sigma$ since $\zeta \in \langle \sigma \rangle_\theta$. Hence $A$ has the $CBS_{\cal K}$-property.

$2 \Longrightarrow 1)$ Let us assume that $A$ has $CBS_{\cal K}$-property. Let $f:A \rightarrow A/\theta$ be an ${\cal A}$-isomorphism, where $\theta \in {\cal K}(A)$, and $\sigma \in [\Delta_A, \theta ]\cap {\cal K}(A)$. Then, by hypothesis, $A/\Delta_A \cong_{\cal A} A \cong_{\cal A} A/\sigma$ and $\Delta_A \in \langle \sigma \rangle_\theta$ (see Eq.(\ref{SIGMAIS})). Thus, we consider the sequence $(\sigma_n)_{n\in {\mathbb{N}}}$ given by    
\begin{align*}
\sigma_0 &= \Delta_A \,,&   \nonumber \\     
\sigma_1 &= \Delta_A \,,&                 \theta_1 &= f_*(\sigma_0) =  \theta/\theta \,,\ \nonumber\\    
\sigma_2 &= u_\theta^{-1}(\theta_1) = \theta \,,& \theta_2 &= f_*(\sigma_1) = f_*(\Delta_A)= \theta/\theta\,,\ \nonumber \\
\sigma_3 &= u_\theta^{-1}(\theta_2) = \theta \,,&          \theta_3 &= f_*(\sigma_2) \,,\ \nonumber \\
         &\vdots                             \,&                   &\vdots               \,\ \nonumber \\
\sigma_{n+1} &= u_\theta^{-1}(\theta_n) \,,&      \theta_{n+1} &= f_*(\sigma_n) \,.\
\end{align*}
By induction, we show that $\sigma_{2n}$ = $\sigma_{2n+1}$ for all $n\geq 1$. Indeed $\sigma_2 = \sigma_3 = \theta/\theta$. Let us suppose that $\sigma_{2k} = \sigma_{2k+1}$. Then,
\begin{eqnarray*}
\sigma_{2(k+1)} &=& u_\theta^{-1}(\theta_{2k+1}) = u_\theta^{-1} f_*(\sigma_{2k}) \\
& = & u_\theta^{-1} f_*(\sigma_{2k + 1}) =  u_\theta^{-1}(\theta_{2(k+1)})\\
& = & \sigma_{2(k+1) + 1}.
\end{eqnarray*}
In this way, $(\sigma_{2n} \lor \neg\sigma_{2n+1})_{n\geq 1} = (\nabla_A, \nabla_A, \nabla_A, \ldots )$ implying that $\sigma_{\Delta_A} = \bigsqcap^{{\cal K}(A)}_{n\geq 1}(\sigma_{2n}\lor\neg \sigma_{2n+1}) = \nabla_A$. Hence $A$ is $CBS_{\cal K}$-complete. 

\qed        
\end{proof}

In rest of the section we study an special framework for the CBS-theorem based on congruences presheaves defined by sets of factor congruences with a Boolean structure. For this we first introduce the following definition.

\begin{definition}\label{BOOLFACC}
{\rm Let ${\cal A}$ be a category of algebras. A {\it Boolean factor congruences presheaf} is a congruences presheaf ${\cal K}$ such that, for each  $A\in {\cal A}$,

\begin{enumerate}

\item
${\cal K}(A) \subseteq FC(A)$.

\item
$\langle {\cal K}(A), \cap, \lor, \neg, \Delta_A, \nabla_A  \rangle$ is a Boolean sublattice of $Con(A)$ where $\neg$ is the factor complement.

\end{enumerate}
}
\end{definition}

By Proposition \ref{CENTER}-2 and Definition \ref{FACTORCPRESHEAF}-3 we can see that for each $\sigma \in {\cal K}(A)$, the Boolean structure of ${\cal K}(A/\sigma)$ is given by 
\begin{equation}\label{AUX5}
\langle {\cal K}(A/\sigma), \lor, \cap, \neg, \Delta_{A/\sigma}, \nabla_{A/\sigma} \rangle \hspace{0.2cm} where \hspace{0.2cm} \neg(\theta/\sigma) = (\neg_\sigma \theta)/\sigma  
\end{equation}

The following proposition allows us to provide examples of Boolean factor congruences presheaf from the center of congruence lattices of the algebras of a category of algebras.

\begin{prop} \label{PRESHEAF1}
Let ${\cal A}$ be a category of algebras such that for each $A \in {\cal A}$ and $\sigma \in Z(Con(A))$, $A/\sigma \in {\cal A}$. Then, the class operator 
${\cal A}\ni A \mapsto Z(Con(A))$ is a congruences operator over ${\cal A}$ and the following assertions are equivalent:

\begin{enumerate}
\item $Z(Con(-))$ is a Boolean factor congruences presheaf.

\item
For each $A \in {\cal A}$, and $\theta \in Z(Con(A))$, $\theta \circ \neg \theta = \nabla_A$ where $\neg \theta$ is the Boolean complement of $\theta$ in $Z(Con(A))$.

\end{enumerate}

\end{prop}

\begin{proof} 
By Proposition \ref{AUX} it is immediate that $Z(Con(-))$ is a congruences operator over ${\cal A}$.

$1\Longrightarrow 2$) Let us assume that $Z(Con(-))$ is a Boolean factor congruences presheaf. Then, for each $A \in {\cal A}$, $Z(Con(A)) \subseteq FC(A)$. Since $Z(Con(A))$ is a Boolean algebra, the complement of an element in $Z(Con(A))$ is unique. Consequently, by condition 2 in Definition \ref{FACTORCPRESHEAF}, for each $\theta \in Z(Con(A))$ we have that $\theta \circ \neg \theta = \nabla_A$.    

$2\Longrightarrow 1$) Let us assume that for each $\theta \in Z(Con(A))$, $\theta \circ \neg \theta = \nabla_A$. Then $Z(Con(A)) \subseteq FC(A)$ for each $A \in {\cal A}$. Let $\sigma \in Z(Con(A))$. By Proposition \ref{CENTER} and Proposition \ref{CORRESP}-2 we have that $\theta \in [\sigma, \nabla_A]\cap Z(Con(A))$ iff $\theta \in Z[\sigma, \nabla_A]$ iff $\theta/\sigma$ in $Z(Con(A/\sigma))$. Thus, by Proposition \ref{PRESHEAF}-3, $Z(Con(-))$ is a congruences presheaf over $ {\cal A}$.  Hence our claim.

\qed
\end{proof}

\begin{example}\label{CONGPER}
{\rm   Let ${\cal A}$ be a congruence permutable variety. Let us notice that for each $A \in {\cal A}$ and $\theta \in Z(Con(A))$, $\theta \cap \neg \theta = \Delta_A$ and $\theta \circ \neg \theta = \theta \lor \neg \theta = \nabla_A$ because the permutability of $\theta$. Then $Z(Con(A)) \subseteq FC(A)$ and, by Proposition \ref{PRESHEAF1}, $Z(Con(-))$ is a Boolean factor congruences presheaf.
}
\end{example}

\begin{example}\label{ARITMET}
{\rm  Let ${\cal A}$ be an arithmetical variety i.e., ${\cal A}$ is a congruence distributive and congruence permutable variety. By Example \ref{CONGPER}, for each $A \in {\cal A}$, $Z(Con(A)) \subseteq FC(A)$ and  $Z(Con(-))$ is Boolean factor congruences presheaf. Since $A$ congruence distributive, $FC(A)$ is a Boolean sublattice of $Con(A)$ and then, $FC(A) \subseteq Z(Con(A))$. Thus $Z(Con(A)) = FC(A)$. In this way $FC(-) = Z(Con(-))$ is Boolean factor congruences presheaf.  Another interesting categories of algebras in which $FC(-) = Z(Con(-))$ is a Boolean factor congruences presheaf are discriminator varieties since they are arithmetical varieties.    
}
\end{example}

\begin{theo}\label{CBS3}
Let ${\cal A}$ be a category of algebras and ${\cal K}$ be a Boolean factor congruences presheaf. Then the following conditions are equivalent:

\begin{enumerate}
\item
$A$ has the $CBS_{{\cal K}}$-property.

\item
For each ${\cal A}$-isomorphism $f:A \rightarrow A/\theta$, where $\theta \in {\cal K}(A)$, and for each $\sigma \in [\Delta_A, \theta ]\cap {\cal K}(A)$ there exists $\zeta \in \langle \sigma \rangle_{\theta}$ and a CBS-sequence $(\sigma_{2n}\lor\neg \sigma_{2n+1})_{n\geq 0}$ {\rm (see  Definition \ref{CBSSEQURNCE})} such that $\sigma_{\zeta} = \bigsqcap^{{\cal K}(A)}_{n\geq 1}(\sigma_{2n}\lor\neg \sigma_{2n+1})$ exists. 

\end{enumerate}
\end{theo}

\begin{proof}
Since for each $A \in {\cal A}$, ${\cal K}(A)$ is a Boolean sublattice of $Con(A)$, for all $\zeta, \sigma \in {\cal K}(A)$ we have that $(\neg \zeta \cap \sigma, \neg (\neg \zeta \cap \sigma )) = (\neg \zeta \cap \sigma, \zeta \lor \neg \sigma)$ is a pair of factor congruences in ${\cal K}(A)$. Hence, by Theorem \ref{CBS2}, our claim.

\qed
\end{proof}

By Theorem \ref{CBS3} and Proposition \ref{INCREASINGSIGMAN2}-2 the next result is immediate.

\begin{prop}\label{CBS33}
Let ${\cal A}$ be a category of algebras, ${\cal K}$ be a Boolean factor congruences presheaf and $A\in {\cal A}$ such that ${\cal K}(A)$ is dual orthogonal $\sigma$-complete. Then $A$ has the $CBS_{{\cal K}}$-property. \qed
\end{prop}

\section{BFC and the $CBS$-property} \label{BFCFC}
Categories of algebras having BFC are examples of categories  in which the class operator $FC(-)$ defines a Boolean factor congruences presheaf. Thus, BFC allows us to establish an interesting  framework for several versions of the CBS-Theorem. Indeed, most of the versions present in literature can be formulated in terms of the $FC(-)$ presheaf.  In this section we deal with this argument and we establish new examples of algebras having the $CBS_{FC}$-property.

\begin{definition}
{\rm 
An algebra $A$ has {\it Boolean factor congruences} (BFC for short) iff $FC(A)$ forms a Boolean sublattice of $Con(A)$.  We say that a {\it category of algebras has BFC} iff each algebra of the category has BFC. 
}
\end{definition}

\begin{prop}\label{PRESHEAF2}
Let ${\cal A}$ be a category of algebras having BFC such that for each $A\in {\cal A}$ and $\sigma \in FC(A)$, $A/\sigma \in {\cal A}$. Then $FC(-)$ is a Boolean factor congruences presheaf. 

\end{prop}

\begin{proof}
Let $A \in {\cal A} $ and $\sigma \in FC(A)$. Let us suppose that $\theta \in FC(A) \cap [\sigma, \nabla_A]$. We want to prove that $\theta /\sigma \in FC(A/\sigma)$. We first note that  $\theta/\sigma \cap (\neg \theta \lor \sigma)/\sigma = \Delta_{A/\sigma}$. Moreover, $\nabla_A = \theta \circ \neg \theta \subseteq  \theta \circ (\neg \theta \lor \sigma)$ and, by Theorem \ref{CORRESP}-3, $\theta/\sigma \circ (\neg \theta \lor \sigma)/\sigma = \nabla_{A/\sigma}$. Thus, $(\theta/\sigma, (\neg \theta \lor \sigma)/\sigma )$ is a pair of factor congruences of $A/\sigma$ and $\theta/\sigma \in FC(A/\sigma)$. 
Now if we suppose  that $\theta/\sigma \in FC(A/\sigma)$ then, by Proposition \ref{BFCQUOT}, $\theta \in FC(A)$. Hence, by Proposition \ref{PRESHEAF}, $FC(-)$ is a Boolean factor congruences presheaf. 

\qed
\end{proof}

The next proposition provides a general method to obtain algebras satisfying the $CBS_{FC}$-property in categories of algebras having BFC.

\begin{prop}\label{DIRCETPRODUCT}
${\cal A}$ be a category of algebras closed by direct products having BFC and $(A_i)_{i\in I}$ be a family of directly indecomposable algebras in ${\cal A}$ then:
$$B = \prod_{i\in I} A_i \hspace{0.2cm} \mbox{satisfies the $CBS_{FC}$-property.} $$ In particular if  $I = {\mathbb N}$ and $A_i = A$ for each $i\in {\mathbb N}$ then $B$ satisfies the $CBS_{FC}$-property in a non trivial way {\rm (see Remark \ref{NONTRIVIALCBS})}.
\end{prop}

\begin{proof}
Note that for each $i\in I$, $FC(A) = \{\Delta_{A_i}, \nabla_{A_i} \}$. Then, by {\rm  \cite[Theorem 2 and Theorem 11]{ISK}}, we can see that $FC(B)$ is lattice isomorphic to $\prod_{i\in I}FC(A_i) = {\bf 2}^I$ where ${\bf 2}$ is the Boolean algebra of two elements. Since ${\bf 2}^I$ is a complete Boolean algebra, by Proposition \ref{CBS33}, $B$ satisfies the $CBS_{FC}$-property. The second part follows by Proposition \ref{NDIRECTPRODUCT}.

\qed
\end{proof}

The rest of the section is devoted to reformulate, in terms of the $CBS_{FC}$-property, several versions of the CBS-theorem  already present in the literature as well as new version of the theorem in categories of algebras having BFC.

\begin{example}
{\rm [Lattice ordered groups] A {\it lattice ordered group} ({\it l-group} for short) is an algebra $\langle A, +, \lor, \land, -, 0 \rangle$ of type $\langle 2,2,2,1,0 \rangle$ such that 
\begin{enumerate}
\item
$\langle A, +,  -, 0 \rangle$ is a group,  \hspace{0.5cm} 3. $x+(s\land t) + y = (x+s+y) \land (x+t+y)$,

\item
$\langle A, \lor, \land \rangle$ is a lattice, \hspace{0.8cm} 4. $x+(s\lor t) + y = (x+s+y) \lor (x+t+y)$.
\end{enumerate}

Thus, l-groups define a variety of algebras denoted by ${\cal LG}$. Let $A \in {\cal LG}$.  If $x\in A$ then we define $\vert x \vert = x \lor -x$. The positive cone of $A$ is given by $A^+ = \{x\in A: x\geq 0 \}$. A set $G \subseteq A$ is said to be orthogonal iff $G \subseteq A^+$ and $x\land y = 0$ for any pair of distinct elements $x,y \in G$. The l-group $A$ is said to be {\it orthogonally $\sigma$-complete} iff each denumerable orthogonal subset of $A$ has supremum in $A$. It is well known that $Con(A)$ is lattice isomorphic to the lattice $I_l(A)$ of all convex normal subgroups (also called l-ideals) of $A$.  Moreover $FC(A)$ is a Boolean sublattice of $Con(A)$ (see \cite[\S XIII-9]{Bir}) identified to a Boolean sublattice of $I_l(A)$, denoted by  $FCI_l(A)$, whose elements are called {\it direct factors} of $A$. Thus, ${\cal LG}$ have BFC and,
by Proposition \ref{PRESHEAF2}, $FC(-)$ is a Boolean factor congruences presheaf. If $I \in FCI_l(A)$ then the set $\neg I$ defined as  $\neg I = \{a\in A: \vert a \vert \land \vert x \vert = 0 \hspace{0.2cm} \mbox{for each $x\in I$} \} $ is the complement of $I$ in $FCI_l(A)$ (see \cite[Eq.(1.3)]{J2}). In order to establish a CBS-theorem for l-groups we need to prove the following result:   
\begin{enumerate}
\item[]
Let $A$ be an orthogonal $\sigma$-complete l-group and $(I_n)_{n \in {\mathbb N}}$ be a dual orthogonal sequence in the Boolean lattice $FCI_l(A)$. Then $\bigcap_{n \in {\mathbb N}} I_n \in FCI_l(A)$.    
\end{enumerate}
Indeed, if $(I_n)_{n \in {\mathbb N}}$ is a dual orthogonal sequence in $FCI_l(A)$ then $(\neg I_n)_{n \in {\mathbb N}}$ is an orthogonal sequence in $FCI_l(A)$ because $FCI_l(A)$ is a Boolean algebra.
By \cite[Lemma 1.5]{J2} $\neg \bigcup _{n \in {\mathbb N}}\neg I_n  \in FCI_l(A)$ and in \cite[Theorem 2.2.5]{SST} it is proved that $\neg \bigcup _{n \in {\mathbb N}}\neg I_n = \bigcap_{n \in {\mathbb N}} \neg \neg I_n = \bigcap_{n \in {\mathbb N}} I_n $. Hence our claim.

Thus, if $A$ is an orthogonal $\sigma$-complete l-group then, by Proposition \ref{CBS33}, $A$ has the $CBS_{FC}$-property. In this example we have reformulated, in terms Boolean factor congruences presheaves, the version of CBS-theorem for l-groups given in \cite{J2}.  

}
\end{example}

\begin{example}
{\rm [$\cal L$-varieties]
$\cal L$-varieties where introduced in \cite{FR} as a general lattice ordered structure in which several versions of CBS-theorem can be formulated. A variety ${\cal A}$ of algebras is a {\it ${\cal L}$-variety} iff 
\begin{enumerate}
\item[(1)]
there are terms of the language of ${\cal A}$ defining on each $A\in {\cal A}$ operations $\lor$, $\land$, 0,1 such that $L(A)=\langle A,\lor,\land,0,1\rangle$ is a bounded lattice;   
\item[(2)]
for all $A\in {\cal A}$ and for all $z\in Z(L(A))$, the binary relation ${\Theta}_z$ on $A$ defined by $a  \Theta_z b$  iff $a\land z = b\land z$  is a congruence on $A$,   such that $A\cong A/{\Theta}_z\times A/{\Theta}_{\neg z}$.   
\end{enumerate}
Example of $\cal L$-varieties are the following (see \cite[\S 2]{FR})
\begin{itemize}
\item
The variety ${\cal L}_{01}$ of bounded lattices and its subvarieties. In particular, distributive lattices and  modular lattices. 

\item
The variety ${\cal LI}_{01}$ of bounded lattices with involution ``$\ninv$" \cite{Ka} satisfying the {\em Kleene equation} $x \land \ninv x = (x \land \ninv x) \land (y \lor \ninv y)$.
Subvarieties of  ${\cal LI}_{01}$ are the variety $\cal OL$ of {\em ortholattices\/} \cite{Bir,MM}, characterized by the equation $x \land \ninv x = 0$, and the variety $\cal K$ of {\em Kleene algebras\/} \cite{BD}, characterized by the distributive law. The intersection $\cal OL \cap K $ is the variety $\cal  B$  of {\it Boolean algebras}. An important subvariety of $\cal OL$ is the variety $\cal OML$ of {\em orthomodular lattices\/} \cite{Bir,MM}.

\item
The variety ${\cal B_\omega}$ of {\it pseudocomplemented distributive lattices} \cite{BD} and the subvariety of Stone algebras ${\cal ST}$ defined as ${\cal ST} = {\cal B_\omega} + \{(x\land y)^* = x^* \lor y^* \}$ where $^*$ is the pseudocomplement (see \cite[\S VIII]{BD}).

\item
The variety ${\cal RL}$ of {\it residuated lattices} \cite{JT} also called {\it
commutative integral residuated $0,1$-lattices} \cite{KO} defined by algebras $ \langle A, \lor, \land, \odot, \rightarrow, 0, 1 \rangle$ of type $ \langle 2, 2, 2, 2, 0, 0 \rangle$ satisfying the following: 
\begin{enumerate}
\item
$\langle A,\odot,1 \rangle$ is an abelian monoid,

\item
$L(A) = \langle A, \lor, \land, 0,1 \rangle$ is a bounded lattice,  

\item
$(x \odot y)\rightarrow z = x\rightarrow (y\rightarrow z)$,

\item
$((x\rightarrow y)\odot x)\land y = (x\rightarrow y)\odot x$,

\item
$(x\land y)\rightarrow y = 1$.

\end{enumerate}
Very important subvarieties of ${\cal RL}$ are: the variety of {\it Heyting algebras} \cite{BD} given by ${\cal H} = {\cal RL} + \{x\odot y = x\land y\}$ and the variety of {\it BL-algebras}, characterized by ${\cal BL} = {\cal RL} +  \{x\land y = x \odot (x\rightarrow y), \hspace{0.2cm} (x\rightarrow y)\lor (y\rightarrow x)= 1\}$. BL-algebras are the algebraic counterpart of the fuzzy logic related to continuous $t$-norm \cite{HAJ}. Important subvarieties of ${\cal BL}$ are: the variety of {\it MV-algebras}, representing the algebraic counterpart of the  infinite-valued \L ukasiewicz logic \cite{CDM, HAJ}, given by ${\cal MV} = {\cal BL} + \{\neg \neg x = x\}$, the variety of {\it linear Heyting algebras}, also known as {\it G\"odel algebras}, given by ${\cal HL} = {\cal H} + \{(x \rightarrow y) \lor (y \rightarrow x) = 1\}$  and the variety of {\it Product logic algebras}, given by ${\cal PL} = {\cal BL} + \{\neg \neg z \odot ((x \odot z)\rightarrow (y \odot z))) \rightarrow (x \rightarrow y) =1, \hspace{0.2cm} x \wedge \neg x = 0\}$.

\item
The varieties of {\it \L ukasiewicz} and of {\it Post algebras} of order $n \geq 2$ {\rm \cite{BD}, as well as the various types of {\it \L ukasiewicz - Moisil} algebras which are considered in \cite{Boi}.}

\item
${\cal PMV}$, the variety of {\it pseudo MV-algebras} \cite{DVU1, GI1}. 
\end{itemize}

Let ${\cal A}$ be a ${\cal L}$-variety. In {\rm \cite[Proposition 1.4]{FR}} it is proved that ${\cal A}$ has BFC where for each $A \in {\cal A}$, $FC(A)$ is Boolean isomorphic to $Z(L(A))$. Then, by  Proposition \ref{PRESHEAF2}, $FC(-)$ is a Boolean factor congruences presheaf. Thus, if $L(A)$ is an orthogonally $\sigma$-complete lattice then, by Proposition \ref{CBS33}, $A$ has the $CBS_{FC}$-property.  Let us notice that the $\sigma$-completeness (orthogonally $\sigma$-completeness) of an algebra $A \in {\cal A}$ does not imply the $\sigma$-completeness (orthogonally $\sigma$-completeness) of  $Z(L(A))$ (see \cite[Example 4.1]{FR}). However, there are  $\cal L$-varieties  where the $\sigma$-completeness (orthogonally $\sigma$-completeness) condition on the algebras  guarantee the corresponding $\sigma$-completeness (orthogonally $\sigma$-completeness) of their centers and then, the $CBS_{FC}$-property. Examples of these particular ${\cal L}$-varieties are: Boolean algebras (where the $CBS_{FC}$-property was obtained by Sikorski and Tarski), Orthomodular lattices (where the $CBS_{FC}$-property was obtained in \cite{DMP}), MV-algebras (where the $CBS_{FC}$-property was obtained in \cite{DMN}), Pseudo MV-algebra (where the $CBS_{FC}$-property was obtained in \cite{J}), Stone algebras {\rm \cite[Proposition 4.3]{FR}}, BL-algebras {\rm \cite[Corollary 4.8]{FR}}, \L ukasiewicz and Post  algebras of order $n$ {\rm \cite[Lemma 3.1]{Cig}}.    
}
\end{example}

\begin{example}
{\rm [Semigroups with $0,1$ and bounded semilattices] A {\it semigroup with $0,1$} is an algebra $\langle A, \cdot, 0,1 \rangle$ of type $\langle 2,0,0 \rangle$  such that the operation $\cdot$ is associative, $0\cdot x = x\cdot 0 =0$ and $1\cdot x = x\cdot 1 =x$. Thus, semigroups with $0,1$ define a variety denoted by ${\cal SG}_{0,1}$. An important subvariety of ${\cal SG}_{0,1}$ is the variety of  {\it bounded semilattices} defined as ${\cal SL}_{0,1} = {\cal SG}_{0,1} + \{x^2 = x, \hspace{0.2cm} x\cdot y = y \cdot x \}$. 
Let ${\cal A}$ be a subvariety of ${\cal SG}_{0,1}$ and $A \in {\cal A}$.  An element $z\in A$ is called {\it central} iff there exist $A_1, A_2 \in {\cal A}$ and a ${\cal SG}_{0,1}$-isomorphism $f: A \rightarrow A_1 \times A_2$ such that $f(z) = (1,0)$. The set of all central elements of $A$ is denoted by $Z(A)$. In \cite{SM0, SM} it is proved that $Z(A)$ is a Boolean algebra in which the meet operation is the semigroup operation. Furthermore, it has been proved that $Z(A)$ is Boolean isomorphic to $FC(A)$. Thus, by Proposition \ref{PRESHEAF2}, $FC(-)$ is a Boolean factor congruences presheaf. Hence, if $A \in {\cal A}$ is an algebra such that $Z(A)$ is orthogonally $\sigma$-complete then, by Proposition \ref{CBS33}, $A$ has the $CBS_{FC}$-property. 
}
\end{example}

\begin{example}
{\rm [Commutative pseudo $BCK$-algebras] A {\it commutative pseudo $BCK$-algebras} ($^{cp}BCK$-algebra for short) \cite{GI1} is an algebra $\langle A, \rightarrow, \rightsquigarrow, 1 \rangle$ of type $\langle 2,2,0 \rangle$ satisfying the following equations: 

\begin{enumerate}
\item
$x\rightarrow (y \rightsquigarrow z) = y\rightarrow (x \rightsquigarrow z)$, \hspace{0.4cm} 4. $(x\rightarrow y) \rightsquigarrow y = (y\rightarrow x) \rightsquigarrow x$,

\item
$x\rightarrow x = x \rightsquigarrow x = 1$, \hspace{1.72cm} 5. $(x \rightsquigarrow y) \rightarrow y = (y \rightsquigarrow x)\rightarrow  x$.

\item
$1\rightarrow x = 1 \rightsquigarrow x = x$,

\end{enumerate}

Thus $^{cp}BCK$-algebras define a variety denoted by $^{cp}{\cal BCK}$. Let $A$ be a $^{cp}BCK$-algebra. The relation $x\leq y$ iff $x \rightarrow y = 1$ iff $x \rightsquigarrow y = 1$ defines a  join semi-lattice order where $x\lor y = (x\rightarrow y) \rightsquigarrow y = (x \rightsquigarrow y) \rightarrow y$. Let us notice that in \cite{KU1} a dually equivalent definition for $^{cp}BCK$-algebras, based on the reverse order, is introduced. In \cite[Corollary 4.4]{EK} it is proved that $^{cp}{\cal BCK}$ is a congruence distributive variety. Then, for each $A \in ^{cp}{\cal BCK}$, $FC(A)$ is a Boolean sub lattice of $Con(A)$. Thus $^{cp}{\cal BCK}$ has BFC and, by Proposition \ref{PRESHEAF2}, $FC(-)$ is a Boolean factor congruences presheaf. By \cite[Lemma 4.1]{KU1} we can dually prove that if $A$ is a dual orthogonal $\sigma$-complete $^{cp}BCK$-algebra then, each dual orthogonal sequences $(\theta_n)_{n\in {\mathbb N}}$ in $FC(A)$ admits infimum $\bigcap_{n\in {\mathbb N}} \theta_n \in FC(A)$. Hence, if $A$ is a dual orthogonal $\sigma$-complete $^{cp}BCK$-algebra then, by Proposition \ref{CBS33}, $A$ has the $CBS_{FC}$-property.
In this example we have reformulated, in terms Boolean factor congruences presheaves, the version of CBS-theorem for $^{cp}BCK$-algebras given in \cite{KU1}.  
}
\end{example}

\begin{example}
{\rm [Church algebras] An algebra $A$ is called {\it Church algebra} \cite{MSA} if there are two constants $0, 1 \in A$ and a ternary term $t(z,x,y)$ called {\it if-then-else term}  in the language of $A$ such that $t(1,x,y) = x$ and  $t(0,x,y) = y$. A variety of algebras ${\cal A}$ is a called a {\it Church variety} iff every algebra in ${\cal A}$ is a Church algebra with respect to the same term $t(z, x, y)$ and constants $0, 1$. Let ${\cal A}$ be a Church variety and $A\in {\cal A}$. An element $e\in A$ is called ${\it central}$ iff the generated congruences $\theta(1,e)$ and $\theta(e,0)$ defines a pair of factor congruences of $A$. We denote by $Z(A)$ the set of al central elements of $A$. It is proved that central elements are equationally characterized in the following way: $e\in A$ is a central element iff whenever   
$a,b \in A$, $\varphi$ is an operation symbol of arity $n$ in the language of ${\cal A}$ and  $\overline{a},  \overline{b} \in A^n$, the following equation are satisfied 
\begin{align*}
t(e,x,x) &= x \,,  & t(e,t(e,x,y),z) &= t(e,x,z) = t(e,x,t(e,y,z)) \,,\   \nonumber \\     
t(e,1,0) &= e  \,,&                 t(e, \varphi^A(\overline{a}),\varphi^A(\overline{b}) ) &= \varphi^A(t(e,a_1, b_1)\ldots t(e,a_n, b_n)  ) . \,\ \nonumber
\end{align*}
Moreover $\langle Z(A), \lor, \land, \neg, 0,1 \rangle$ where $x\lor y = t(x,1,y)$, $x\land y = t(x,y,0)$ and $\neg x = t(x,0,1)$ is a Boolean algebra isomorphic to $FC(A)$. Thus, ${\cal A}$ has BFC and, by Proposition \ref{PRESHEAF2}, $FC(-)$ is a Boolean factor congruences presheaf. In what follows we shall study concrete examples of Church algebras satisfying the $CBS_{FC}$-property.

\begin{itemize}
\item
{\it Rings with identity} define a Church variety denoted by ${\cal R}_1$ where the if-then-else term is given by $t(z,x,y) = (y+z-zy)\cdot(1-z+zx)$. If $A\in {\cal R}_1$ then $Z(A)$ is the set of  central idempotent elements of $A$. Two interesting examples of rings with identity whose its central idempotent elements define a complete Boolean algebra are

\begin{itemize}
\item[-]
{\it Division rings} because they are simple algebras. Then, by Proposition \ref{DIRCETPRODUCT}, denumerable direct products of division rings satisfy the  $CBS_{FC}$-property in a non trivial way. 

\item[-]  
{\it Baer rings} i.e., a ring with identity $A$ such that for every subset $S \subseteq A$ the right annihilator $Ann_r(S) = \{r\in A: \forall s\in S, r\cdot s = 0 \}$ is the principal right ideal generated by an idempotent. In \cite[\S 3, 3.3 ]{Ber1} it is proved that $Z(A)$ is a complete Boolean algebra. Then, by Proposition \ref{CBS33}, Baer rings have the $CBS_{FC}$-property.

\end{itemize} 

\item
{\it $*$-Rings}. They are rings with identity having an involution operation $x\mapsto x^*$ such that $x^{**} = x$, $(x+y^*) = x^*+y^*$ and $(x\cdot y)^* = y^* \cdot x^*$. By the underling ring with unity structure, $*$-rings define a Church variety denoted by ${\cal R}_1^*$. Examples of $*$-rings having the  $CBS_{FC}$-property are the Baer $*$-rings. Indeed: A {\it Baer $*$-ring} is a $*$-rings $A$ such that for every subset $S \subseteq A$, the right annihilator $Ann_r(S) = eA$ where $e$ is a projection (i.e. $e^2 = e^* = e$). By \cite[P18, 4A]{Ber2} we can see that $Z(A)$ is determined by the central projections. Moreover in a Baer $*$-rings its central projections form a complete Boolean algebra \cite[p.30, Corollary ]{KAP}. Thus, by Proposition \ref{CBS3}, Baer $*$-rings have the $CBS_{FC}$-property.

\end{itemize} 
}
\end{example}

\begin{example}
{\rm [Effect and Pseudo-effect algebras] 
Although there are versions of the CBS-theorem related to these structures \cite{JEN, DVU4}, from a strictly formal point of view, these versions can not be framed in our formalism because these algebras are defined by a binary partial operation. However, we can easily extend the notion of Boolean factor congruences presheaf and the CBS-property to these particular algebraic structures.
A {\it pseudo-effect algebra} is a partial algebra $\langle E, +, 0,1 \rangle$ of type $\langle 2,0,0 \rangle$ such that

\begin{enumerate}
\item
$a+b$ and $(a+b)+c$ exist iff $b+c$ and $a+(b+c)$ exist and in this case $(a+b)+c = a+(b+c)$,

\item
for each $a\in E$ there is exactly one $a^- \in E$ and exactly one $a^\sim \in E$ such that $a^- + a = a + a^\sim = 1$,

\item
if $a+b$ exists, there are elements $d,e \in E$ such that $a+b = d+a = b+e$,  

\item
if $1+a$ or $a+1$ exists then $a=0$.

\end{enumerate}

We denote by ${\cal PE}$ the category whose objects are pseudo-effect algebras and whose arrows, called {\it ${\cal PE}$-homomorphisms}, are functions $f:E \rightarrow F$ between pseudo-effect algebras such that $f(0) = 0$, $f(1) = 1$ and $f(a+b) = f(a) + f(b)$ whenever $a+b$ exists in $E$. If $+$ is commutative then $E$ is said to be an {\it effect algebra} and we denote by ${\cal E}$ the sub category of effect algebras. Let $E\in {\cal PE}$. If we define $a\leq b$ iff there exists $x\in E$ such that $a+x = b$ then $\langle E, \leq \rangle$ is a partial order such that $0\leq a \leq 1$ for any $a\in E$. For a given $e\in E$ the interval $[0,e]$ endowed with $+$ restricted to $[0,e]^2$ is a pseudo effect algebra $[0,e] = \langle [0,e], +, 0,e \rangle$. An element $e\in E$ is said to be $\it central$ iff  there exists a ${\cal PE}$-isomorphism $f_e: E \rightarrow [0,e] \times [0,e^\sim]$ such that $f_e (e) = (e,0)$ and, if $f_e(x) = (x_1, x_2)$ then $x = x_1 + x_2 = x_1 \lor x_2$. We denote by $Z(E)$ the set of all central elements of $E$. In  \cite[Proposition 2.2]{DVU4} it is proved that, for any $x\in E$ and $e\in Z(E)$, $x \land e \in E$, $x \land e^\sim \in E$, $f_e(x) = (x\land e, x \land e^\sim)$ and $\pi_e:E \rightarrow [0,e]$ such that $\pi_e(x) = x\land e$ is a surjective ${\cal PE}$-homomorphism. Furthermore, in
\cite[Theorem 2.3]{DVU4}, it is proved that $\langle Z(E),\land, ^\sim, 0,1 \rangle$ is a Boolean algebra.  Let us notice that for each $e\in Z(E)$, $\theta_e = \{(x,y) \in E^2: x\land e = y\land e \}$ defines a congruence on $E$ such that $E/\theta_e \cong_{_{{\cal PE}}} [0,e]$. Let us consider the set $FC(E) = \{\theta_e: e\in Z(E)\}$. It is not very hard to see that for each $e_1, e_2 \in Z(E)$, $\theta_{e_1} \cap \theta_{e_2} = \theta_{e_1 \lor e_2}$. Moreover the ordered set $\langle FC(E), \subseteq \rangle$ defines a Boolean algebra $\langle FC(E), \cap, \lor, \neg, \Delta_E \nabla_E \rangle$ where, $\theta_{e_1} \lor \theta_{e_2} = \theta_{e_1 \land e_2}$, $\neg \theta_e = \theta_{e^\sim}$ and the function $e\mapsto \theta_e$ is an order reverse isomorphism from $Z(E)$ to $FC(E)$. We also note that the class operator $E \mapsto FC(E)$ defines a congruence operator over ${\cal PE}$ in the sense of Definition \ref{congoperator} and the class $Hom_{{\cal PE}_{FC}}$ (see Eq.(\ref{HOMAK})) can be described in the following way:
\begin{equation}\label{HOMPE}
Hom_{{\cal PE}_{FC}} = \bigcup_{E\in {\cal PE}}\{E \stackrel{f_e}{\rightarrow} [0,e]: f_e(x) = x\land e \hspace{0.2cm} and \hspace{0.2cm} e\in Z(E) \}.
\end{equation}
In \cite[Proposition 2.8]{DVU4} it is proved that: 
\begin{equation}\label{ZPE}
\mbox{for each $e \in Z(E)$ and $x\leq e$, $x\in Z([0,e])$ iff $x\in Z(E)$}.
\end{equation}
From Eq.(\ref{ZPE}), it immediately follows that $Hom_{{\cal PE}_{FC}}$ is closed by composition of ${\cal PE}$-homomorphisms and then ${\cal PE}_{FC} = \langle Ob({\cal PE}), Hom_{{\cal PE}_{FC}}\rangle$ defines a category. Let us notice that Eq.(\ref{ZPE}) also implies that if $E \stackrel{f_e}{\rightarrow} [0,e] \in Hom_{{\cal PE}_{FC}} $ and $\theta_a \in FC([0,e])$ then $[FC(f_e)](\theta_a) = f_e^*(\theta_a) = \{(x,y) \in E^2 : x\land a = y \land a \} \in FC(E)$.  Consequently, it is not very hard to see that $FC: {\cal PE}_{FC} \rightarrow Set$ is a presheaf. Thus, following Definition \ref{BOOLFACC}, we can refer to $FC(-)$ as a Boolean factor congruences presheaf. Taking into account the CBS type theorems for pseudo-effect algebras and effect algebras (see \cite[Theorem 6.3]{DVU4} and \cite[Theorem 1.6]{JEN} respectively) and the order reverse identification $Z(E) \cong FC(E)$ for each $E \in {\cal PE}$, the $CBS_{FC}$-property for these partial structures reads:
\begin{itemize}
\item[] A pseudo-effect algebra $E$ has the $CBS_{FC}$-property iff the following holds: Given a pseudo-effect algebra $F$, and $\theta_f \in FC(F)$ such that there is $\theta_e \in FC(E)$ with $E \cong_{_{{\cal PE}}} F/\theta_f$ and $F \cong_{_{{\cal PE}}} E/\theta_e$ it follows that $E \cong_{_{{\cal PE}}} F$.   
\end{itemize}
In \cite[Proposition 6.2]{DVU4} it is proved that if $E,F \in {\cal PE}$ and $h:E \rightarrow [0,f]$ is a ${\cal PE}$-isomorphism where $f\in Z(F)$ then, for each $e\in Z(E)$, $h(e) \in Z(F)$. This result and the order reverse identification $Z(E) \cong FC(E)$ allows us to establish the useful equivalence of the $CBS_{FC}$-property given in Theorem \ref{CBS1}. Finally, in \cite[Theorem 1.6]{JEN} and \cite[Theorem  6.3]{DVU4}, $\sigma$-completeness type conditions imposed on the center of an effect algebra or a pseudo-effect algebra imply $CBS_{FC}$-property in ${\cal E}$ and ${\cal PE}$ respectively. In this way, we have extended our abstract framework for the CBS-theorem to two partial algebraic structures.     
}
\end{example}

\end{document}